\documentclass[10pt,english]{amsart}
\usepackage{amsfonts,amsmath,amssymb,latexsym}

\newtheorem{theorem}{Theorem}[section]
\newtheorem{proposition}{Proposition}[section]
\newtheorem{lemma}{Lemma}[section]
\newtheorem{definition}{Definition}[section]
\newtheorem{remark}{Remark}[section]

\usepackage[colorlinks=true]{hyperref}

\title[Attaching handles to Delaunay nodo\"{\i}ds]{Attaching handles to Delaunay nodo\"{\i}ds}

\author{Frank Pacard}
\address{\'Ecole Polytechnique, Centre de Math\'ematiques Laurent Schwartz, UMR-CNRS 7640, and Institut Universitaire
Palaiseau, 91128 France}
\email{frank.pacard@math.polytechnique.fr}

\author{Harold Rosenberg}
\address{Instituto de Matematica Pura y Aplicada, 110 Estrada Dona Castorina, Rio de Janeiro 22460-320, Brazil}
\email{hrosen@free.fr}

\thanks{Acknowledgments : The second author is partially supported by the
ANR-08-BLANC-0335-01 grant.}

\begin{document}

\maketitle

\begin{abstract}
For all $m \in \mathbb N - \{0\}$, we prove the existence of a one dimensional
family of genus $m$, constant mean curvature (equal to $1$) surfaces which are
complete, immersed in $\mathbb R^3$ and have two Delaunay ends asymptotic to
nodo\"{\i}dal ends. Moreover, these surfaces are invariant under the group of
isometries of $\mathbb R^3$ leaving a horizontal regular polygon with $m+1$
sides fixed.
\end{abstract}

\section{Introduction}

Delaunay surfaces are complete, non compact constant mean curvature surfaces
of revolution in $\mathbb R^3$ which are either embedded or immersed. The
embedded Delaunay surfaces are usually referred to as {\em undulo\"{\i}ds}. The
elements of this family are generated by {\it roulettes} of ellipses \cite{Eel} and they
interpolate between a right cylinder $S^1 ( \frac 1 2) \times \mathbb R \subset
\mathbb R^3$ and a singular surface which is constituted by infinitely many tangent
spheres of radius $1$ which are periodically arranged along the vertical axis.
Close to the singular limit, the Delaunay undulo\"{\i}ds can be understood as
infinitely many spheres of radius $1$ which are disjoint, arranged periodically
along the vertical axis ; each sphere being connected to its two nearest neighbors
by cateno\"{\i}ds whose rotational axis is the vertical axis, which have been scaled by a small
factor $\tau >0$.

The immersed Delaunay surfaces are referred to as {\em nodo\"{\i}ds}. The
element of this family are generated by {\it roulettes} of hyperbola \cite{Eel}. Again,
part of this family converges to infinitely many spheres of radius $1$ which are
periodically arranged along the vertical axis. In contrast to undulo\"{\i}ds, close to
the singular limit, the Delaunay nodo\"{\i}ds can be understood as infinitely many
spheres of radius $1$ which are either disjoint or slightly overlapping and which
are arranged periodically along the vertical axis ; each sphere being connected to
its two nearest neighbors (with whom it shares a slight overlap) by cateno\"{\i}ds
whose axis is the vertical, which have been scaled by a small factor $\tau >0$.

In this paper, we prove the existence of constant mean curvature surfaces which
have two Delaunay ends (of nodo\"{i}d type) and finite genus.

\begin{theorem}
For all $m \geq 1$, there exists a one parameter family of genus $m$ constant
mean curvature (with mean curvature equal to $1$) surfaces which are invariant
under the action of the full dihedral group ${\rm Dih}_{m+1}^{(3)}$ (the group of
isometries of $\mathbb R^3$ leaving a horizontal regular polygon with $m+1$
sides fixed) and which have two Delaunay ends asymptotic to nodo\"{\i}dal ends.
\label{th:main}
\end{theorem}

Let us briefly describe how these surfaces are constructed since this will be the
opportunity to give a precise picture of the surfaces themselves. As already
mentioned, close to the singular limit, the Delaunay nodo\"{\i}ds can be understood
as infinitely many spheres of radius $1$ which are either disjoint or slightly
overlapping, arranged periodically along the vertical axis and which are
connected together by cateno\"{\i}ds with vertical axis, which are scaled by a small
factor $\tau >0$, these latter are called {\em cateno\"{\i}dal necks}. The spheres of
radius $1$ arranged along the vertical axis can be ordered (by the height of their
center) and can be indexed by $j \in \mathbb Z$ (without loss of generality, we can
assume that the center of the sphere of index $j$ is at height $2j+1$). In this
description, one can check that the distance between the centers of two
consecutive spheres can be expanded as
\[
d_\tau = 2 + 2 \, \tau \, \log \tau + \mathcal O (\tau) ,
\]
as $\tau$ tends to $0$. In order to obtain the surfaces of Theorem~\ref{th:main},
instead of connecting the sphere indexed by $0$ and the sphere indexed by $1$
using {\em one} cateno\"{\i}dal neck, we connect these two spheres using $m
+1$ cateno\"{\i}ds which are scaled by a factor
\[
\tilde \tau = \frac{\tau}{m+1} + \mathcal O (\tau^{3/2}),
\]
and whose axis are vertical and pass through the vertices of a horizontal regular
polygon (with $m+1$ sides) of size $\rho >0$. We will show that this construction is
successful provided the parameter $\rho$ which measures the size of the polygon,
is carefully chosen (as a function of $\tau$) and, in fact, we will find that
\[
\rho^2 = {\frac{m}{m+1} \, \frac{\tau}{2}} + \mathcal O (\tau^{5/4}).
\]
Notice that all the surfaces we construct have the same small vertical flux (we refer to
\S 4 for a definition of the flux of a Delaunay surface).

Our construction is quite flexible and provides many other interesting constant
mean curvature surfaces. For example, using similar ideas and proofs, one can
also construct singly periodic constant mean curvature surfaces with (infinite)
topology~: starting with the spheres of radius $1$ which are periodically arranged
along the vertical axis and which are either disjoint or slightly overlapping, we can
choose to connect {\em any} two consecutive spheres using $m+1$ cateno\"{\i}ds
scaled by a factor $\tilde \tau$ whose axis are vertical and pass through the
vertices of a horizontal regular polygon (with $m+1$ sides) of size $\rho >0$. More
generally, there is strong evidence that the following is true :

It should be possible to construct constant mean curvature surfaces starting from a
subset $\mathfrak Z \subset \mathbb Z$ and assuming that, for all $j \in \mathbb Z -
\mathfrak Z$, we decide to connect the sphere of index $j$ to the sphere of index
$j+1$ using {\em one} cateno\"{\i}d whose axis is the vertical axis and which is
scaled by a factor $\tau$, while, when $j \in \mathfrak Z$, we decide to connect the
sphere of index $j$ to the sphere of index $j+1$ using $m+1$ cateno\"{\i}ds
whose axis are vertical and pass through the vertices of a small horizontal regular
polygon (with $m+1$ sides) of size $\rho >0$ with $\rho^2 \sim {\frac{m}{m+1} \,
\frac{\tau}{2}} $, and which are scaled by a factor $\tilde \tau \sim \frac{\tau}{m+1}$.
We believe that this configuration can be perturbed into a genuine constant mean
curvature surface.

To complete this introduction, let us mention that the present construction is very
much inspired by \cite{Hau-Pac} where the construction of minimal surfaces in $
\mathbb R^3$ which have finite genus and two Riemann type ends is performed.
In fact, part of the analysis in the present paper parallels the analysis in
\cite{Hau-Pac}. Nevertheless, in the present situation, some extra technical
difficulties arise in the construction (see \S 6) since the points where the connected
sum is performed are located at the vertices of a polygon whose size tends to $0$
as the parameter $\tau$ tends to $0$.

We end this introduction by giving an overview of the paper. In section 2 we recall
some well known facts about the mean curvature operator of normal graphs with
special emphasize on the differential of the mean curvature operator. Section 3 is
concerned with harmonic extensions on half cylinders for which we prove some
decay properties. The next section is quite long, it starts with a careful description
of the Delaunay nodo\"{\i}ds as the Delaunay parameter $\tau$ tends to $0$ (i.e.
close to the singular limit). Then, we proceed with the analysis of the Jacobi
operator about a Delaunay surface as the Delaunay parameter tends to $0$.
Finally, in section 4.5, we apply the implicit function theorem about a half
nodo\"{\i}d (which is a constant mean curvature surface with one boundary and
one Delaunay end) to prove the existence of an infinite dimensional family of
constant mean curvature surfaces which have one Delaunay end and one
boundary. These surfaces are close to the half nodo\"{\i}d we started with and are
parameterized by their boundary data. In section 6, we perform a similar analysis
starting from the cateno\"{\i}d. As a result, we obtain the existence of an infinite
dimensional family of constant mean curvature surfaces which have two
boundaries, are close to a truncated cateno\"{\i}d and are parameterized by their
boundary data. In section 6, we start with a unit sphere from which we excise one
small disc close to the north pole and $m+1$ small discs arranged symmetrically
at the vertices of a regular polygon near the south pole. We perturb this surface
with $m+2$ boundaries applying the implicit function theorem to obtain an infinite
dimensional family of constant mean curvature surfaces which are parameterized
by their boundary data. In the final section, we explain how all these pieces can be
connected together to produce the surfaces in Theorem~\ref{th:main}. At this stage,
the problem then reduces to be able to chose the boundary data of the different
summands so that their union is a $\mathcal C^1$ surface, since elliptic regularity
theory will imply that what we have built is a smooth constant mean curvature
surface.

The construction heavily relies on the analysis of elliptic operators on non compact
spaces as in \cite{Mel}, \cite{Maz}, \cite{McO}. It is true that similar techniques and
ideas have already been used in many constructions, but the proofs are usually
hard to read for non specialists since they always refer to results which are difficult
to find in the literature in the precise form they are needed. This is the reason why
we have decided to present here complete proofs based on simple well known
tools, hoping that this will help the interested reader to master these technics.

Finally, we mention a problem related to our work. To introduce this problem,
we consider $\Sigma$ to be the union of the upper hemisphere of the sphere of
radius $1$ centered at the points $(0,0,-1)$ and the lower hemisphere of the sphere
of radius $1$ centered at the points $(0,0,1)$. The existence of undulo\"{\i}ds,
nodo\"{i}ds with small Delaunay parameters and the existence of the surfaces
we construct in this paper show that, for all $\epsilon >0$ there exists infinitely
many constant mean curvature ($=1$) surfaces which are included in an $\epsilon$-tubular
neighborhood of the undulo\"{\i}d and are not congruent. Obviously a similar
result holds for the surface $\Sigma$.

Now, if we consider two radius one spheres tangent at a point. Can one find
constant mean curvature ($=1$) surfaces (with no boundary) in any small tubular
neighborhood of this configuration~? In fact, we can not even answer the
(apparently) simpler but striking question. Is there any compact mean curvature ($=1$)
surface (with no boundary) near a radius one sphere? More precisely~: is there
an $\epsilon_0>0$, such that if $\Sigma$ is a mean curvature ($=1$)
surface in the $\epsilon_0$-tubular neighborhood of a radius one sphere,
then $\Sigma$ is congruent to the sphere~? In other words, what is the {\em form}
of a compact constant mean curvature surface~?

\section{Generalities}

\subsection{The mean curvature}

We gather some basic material concerning the mean curvature of a surface in
Euclidean space. All these results are well known but we feel that it makes the
reading of the paper easier if we collected them here. Moreover, this will also be
the opportunity to introduce some of the notations we will use throughout the
paper. We refer to \cite{Col-Min} or \cite{Law} for further details.

Let us assume that $\Sigma$ is a surface which is embedded in $\mathbb R^3$.
We denote by $g$ the metric induced on $\Sigma$ by the Euclidean metric $
\mathring g$ and by $h$ the second fundamental form defined by
\[
h ( t_1, t_2) = - \mathring g \, (\nabla_{t_1} N, t_2) ,
\]
for all $t_1, t_2 \in T \Sigma$. Here $N$ is a unit normal vector field on $\Sigma$.
In this paper, we agree that the mean curvature of a surface is defined to be the
{\em average} of the principal curvatures, or, since we are interested in $2$
dimensional surfaces, the half of the trace of the second fundamental form. Hence,
the mean curvature of $\Sigma$ is given by
\[
H : = \frac{1}{2} \, {\rm tr}^g h ,
\]
and the mean curvature vector is then given by $\vec{H} : = H\, N$.

For computational purposes, we recall that the mean curvature appears in the first
variation of the area functional. More precisely, given $w$, a sufficiently small
smooth function which is defined on $\Sigma$ and has compact support, we
consider the surface $\Sigma_w$ which is the normal graph over $\Sigma$ for the
function $w$. Namely
\[
\Sigma \ni p \longmapsto p + w (p) \, N(p) \in \Sigma_w .
\]
We denote by $A_w$ the area of the surface $\Sigma_w$ (we assume that this
area is finite). Then
\[
D A_{|w=0} (v) = - 2 \, \int_\Sigma H\, v \, {\rm dvol}_g .
\]
In the case where surfaces close to $\Sigma$ are parameterized as graphs over $
\Sigma$ using a vector field $\tilde N$ which is transverse to $\Sigma$ but which is
not necessarily a unit normal vector field, the previous formula has to be modified.
Let us denote by $\tilde \Sigma_w$ the surface which is the graph over $\Sigma$,
using the vector field $\tilde N$, for some sufficiently small smooth function $w$.
Namely
\[
\Sigma \ni p \longmapsto p + w (p) \, \tilde N(p) \in \tilde \Sigma_w .
\]
We denote by $\tilde A_w$ the area of this surface. The previous formula has to be
changed into
\begin{equation}
D \tilde A_{|w=0} (v) = - 2 \, \int_\Sigma (\vec H \cdot \tilde N) \, v \, {\rm dvol}_g .
\label{eq:na}
\end{equation}

In the next result, we give the expression of the mean curvature $H_w$ of the
surface $\Sigma_w$ in terms of $w$. Some notations are needed. For $z \in
\mathbb R$ small enough, we define $g_z$ to be the induced metric on the
parallel surface
\[
\Sigma_z : = \Sigma + z \, N .
\]
It is given explicitly by
\[
g_z = g - 2 \, z \, h + z^2 \, k ,
\]
where the tensor $k$ is defined by
\[
k \, (t_1, t_2) : = g (\nabla_{t_1} \, N, \nabla_{t_2} \, N) .
\]
for all $t_1, t_2 \in T\Sigma$. With these notations, we have the~:
\begin{proposition}
The mean curvature $H_w$ of the surface $\Sigma_w$ is given by the formula
\[
\begin{array}{rlllll}
H_w & = & \displaystyle \Big[ \frac 12 \, \sqrt{1+|\nabla^{g_z} w|^2} \, {\rm tr}^{g_z}
(h - w \, k) + \frac{1}{2} \, {\rm div}_{g_z } \, \Big( \frac {\nabla^{g_z} w}{\sqrt{1 + |
\nabla^{g_z} w|^2}}\Big) \\[5mm]
& - & \displaystyle \frac 12 \, \frac {1}{\sqrt{1+|\nabla^{g_z} w|}} \, \left( h - w \, k
\right) (\nabla^{g_z} w , \nabla^{g_z} w) \, \Big]_{| z=w} .
\end{array}
\]
\label{pr:mcng}
\end{proposition}
\begin{proof} The induced metric $\tilde g$ on $\Sigma_w$ is given by
\[
\tilde g = g_{z =w} + d w \otimes dw .
\]
In particular, this implies that
\[
{\rm det} \, \tilde g = \left(1+ |\nabla^{g_w} w|^2 \right) \, {\rm det} \, g_w .
\]
We can now compute the area of $\Sigma_w$
\[
A_w = \int_\Sigma \sqrt{1+|\nabla^{g_w} w|^2} \, {\rm dvol}_{g_w} ,
\]
as well as the differential of this functional with respect to $w$. In doing so, one
should be careful that the function $w$ appears implicitly in the definition of
$g_w$. We find using an integration by parts
\[
\begin{array}{rllllll}
DA_w (v) & = & - \displaystyle \int_\Sigma {\rm div}_{g_w} \left(
\frac {\nabla^{g_w} w}{\sqrt{1 + |\nabla^{g_w} w|^2 }} \right) \, v \, {\rm dvol}_{g_w}
\\[3mm]
& - & \displaystyle \frac{1}{2} \, \int_\Sigma \frac {1}{\sqrt{1 + |\nabla^{g_w} w|^2 }}
\, g'_w (\nabla^{g_w} w, \nabla^{g_w} w) \, \, v \, {\rm dvol}_{g_w} \\[3mm]
& + & \displaystyle \frac{1}{2} \, \int_\Sigma \sqrt{1 + |\nabla^{g_w} w|^2} \, {\rm
tr}^{g_w} \, g'_w \, \, v \, {\rm dvol}_{g_w} .
\end{array}
\]
where $g'_w : = \partial_z g_z \, _{|z=w} = - 2 \, (h - w \, k)$. To proceed, observe
that, if $N_w$ denotes the normal vector field about $\Sigma_w$, we have
\[
N_w= \frac{1}{\sqrt{ 1+ |\nabla^{g_w} w|^2}}æ\, \left( N - \nabla^{g_w} w \right),
\]
and hence we get
\[
{\rm dvol}_{g_w} = ( N_w \cdot N) \, {\rm dvol}_{\tilde g} .
\]
The result then follows at once from (\ref{eq:na}).
\end{proof}

\subsection{Linearized mean curvature operators}

Again the material in the section is well known and we refer to \cite{Col-Min} and
\cite{Law} for a more detailed description. The Jacobi operator appears in the
linearization of the mean curvature operator when nearby surfaces are
parameterized as normal graphs over a given surface. It follows from Proposition~
\ref{pr:mcng} that the differential of $w \longmapsto H_w$ with respect to $w$,
computed at $w =0$, is given by
\[
J : = D H_{w=0} = \frac{1}{2} \, \left( \Delta_{g} + {\rm tr}^{g} k \right) ,
\]
where $\Delta_g $ is the Laplace-Beltrami operator on $\Sigma$ and ${\rm tr}^{g}
k$ is the square of the norm of the shape operator.

Finally, we recall that if $\Xi$ is a Killing vector field (namely $\Xi$ generates a one
parameter family of isometries) then the function $N \cdot \Xi$, which is usually
referred to as a {\em Jacobi field}, satisfies
\[
J \, ( N \cdot \Xi) = 0 .
\]

This is probably a good time to recall some elementary facts concerning linearized
mean curvature operators when different vector fields are used. As above, we
assume that we are given a vector field $\tilde N$ which is transverse to $\Sigma$,
but which is not necessarily a unit normal vector field. Any surface close enough to
$\Sigma$ can be considered either as a {\em normal} graph over $\Sigma$ or as
a graph over $\Sigma$, using the vector field $\tilde N$, hence, we can define two
nonlinear operators
\[
w \longmapsto H_w , \qquad \mbox{\rm and} \qquad w \longmapsto \tilde H_{w} ,
\]
which are (respectively) the mean curvature of the {\em normal} graph of $w$ and
the mean curvature of the graph of $w$ using the vector field $\tilde N$. The
following result gives the relation between the differentials of these two operators
at $w=0$ \cite{Maz-Pac-Pol}.
\begin{proposition}
The following relation holds
\[
D\tilde H_{|w=0} (v) = DH_{|w=0} ( (\tilde N \cdot N) \, v ) + (\nabla H \cdot \tilde N )
\, v ,
\]
where $H$ denotes the mean curvature of $\Sigma$. In the particular case where
$\Sigma$ has constant mean curvature, this formula reduces to
\[
D \tilde H_{|w =0} (v) = DH_{|w=0} ( (\tilde N \cdot N)\, v) .
\]
\label{pr:2.2}
\end{proposition}
\begin{proof}
The implicit function theorem can be applied to the equation
\[
p + t \, N(p) = q + s \, \tilde N (q) ,
\]
to express (at least locally) $q$ and $s$ as functions of $p$ and $t$, namely
\[
q = \Phi (p,t) \, \qquad \mbox{and} \qquad s = \Psi (p,t) ,
\]
with $\Phi (p,0)=p$ and $\Psi (p,0)=0$. It is easy to check that
\[
\partial_t \Phi (\cdot , 0) = - \frac{1}{ \tilde N \cdot N }\, \tilde N^T, \qquad
\mbox{and} \qquad \partial_t \Psi (\cdot ,0) = \frac{1}{ \tilde N \cdot N}æ.
\]
where the superscript $^T$ denotes the projection over $T\Sigma$.

Differentiation of the identity
\[
\tilde H_{\Psi ( \cdot , w )} (\Phi (p, w (p) ) = H_w (p) ,
\]
with respect to $w$, at $w =0$, we find
\[
D \tilde H_{|w =0} ( \partial_t \Psi ( \cdot , 0 ) \, v ) + \nabla \tilde H_{|w = 0} \cdot
\partial_t \Phi \, v = DH_{|w=0} (v) .
\]
The result then follows from the expression of $\partial_t \Phi$ and $\partial_t \Psi$
and the fact that $\tilde H_{|w =0} = H_{|w=0}$. \end{proof}

\section{Harmonic extensions}

For all $x \in \mathbb R^2$ and all $r >0$ we denote by $D(x,r)\subset \mathbb
R^2$ the open disc of radius $r $, centered at $x$ and $\overline D(x,r)\subset
\mathbb R^2$ the closed disc of radius $r $, centered at $x$. In this section, we
study the harmonic extension either in a half cylinder $[0, \infty) \times S^1$, the
punctured unit disc $\overline D^*(0,1)$ in $\mathbb R^2$ or the complement of
the closed unit disc $\mathbb R^2 - D(0,1)$, of a function which is defined on the
unit circle $S^1$. We will use the fact that all these domains are conformal to each
other and that the Laplacian is conformally invariant in dimension $2$.

Let us assume that we are given a function $f\in \mathcal C^{2, \alpha} (S^1)$. We
consider $F$ to be the bounded harmonic extension of $f$ in the half cylinder,
endowed with the cylindrical metric
\[
g_{cyl} = ds^2 + d \theta^2 .
\]
In other words, $F$ is bounded and is a solution of
\[
\Delta_{g_{cyl}} \, F = 0 ,
\]
in $[0, \infty) \times S^1$ with $F= f$ on $\{0\} \times S^1$.

Observe that one can use cylindrical coordinates to parameterize the punctured
unit disc by
\[
\tilde X(s, \theta ) = (e^{-s} \, \cos \theta, e^{-s} \, \sin \theta) ,
\]
in which case the function $\tilde F$ defined by $\tilde F \circ \tilde X : = F $ is the
unique bounded solution of
\[
\Delta \, \tilde F = 0 ,
\]
(where $\Delta$ denotes the Laplacian in $\mathbb R^2)$ in the punctured unit
disc with $\tilde F = f$ on $S^1$. We set
\[
W^{\rm ins}_f : = \tilde F .
\]

Also, one can use cylindrical coordinates to parameterize the complement of the
unit disc in $\mathbb R^2$ by
\[
\hat X(s, \theta ) = (e^{s} \, \cos \theta, e^{s} \, \sin \theta) ,
\]
in which case $\hat F$ defined by $\hat F \circ \hat X = F$ is the unique bounded
solution of
\[
\Delta \, \hat F =0 ,
\]
in the complement of the unit disc with $\hat F = f$ on $S^1$. We set
\[
W^{\rm out}_f : = \hat F .
\]
In particular, all properties of $F$ will transfer easily to $\tilde F$ and $\hat F$.

Given a function $f $ defined on $S^1$, we shall frequently assume that one or both
of the following assumptions is/are fulfilled
\[
(H1) \qquad \qquad \quad \qquad \qquad \qquad \int_{S^1} \, f \, d \theta = 0 ,
\]
and
\[
(H2) \qquad \int_{S^1} \cos \theta \, f \, d \theta = \int_{S^1} \sin \theta \, f \, d \theta
= 0 .
\]

The following result follows essentially from \cite{Fak-Pac} where a similar result
was proven in higher dimensions~:
\begin{lemma}
There exists a constant $C >0$ such that, for all $f\in \mathcal C^{2, \alpha} (S^1)$
satisfying (H1), we have
\[
\| e^{s} \, F \|_{\mathcal C^{2, \alpha} ([0, \infty) \times S^1)} \leq C \, \| f \|_{\mathcal
C^{2, \alpha} (S^1)} ,
\]
and, if $f$ satisfies (H1) and (H2), we have
\[
\| e^{2s} \, F \|_{\mathcal C^{2, \alpha} ([0, \infty) \times S^1)} \leq C \,
\| f \|_{\mathcal C^{2, \alpha} (S^1)} .
\]
\end{lemma}

Before we proceed with the proof of this result, let us emphasize that the norms in
${\mathcal C^{2, \alpha} ([0, \infty) \times S^1)}$ are computed with respect to the
cylindrical metric $g_{cyl}$.

\begin{proof}
We consider the Fourier series decomposition of the function $f$
\[
f (\theta) = \sum_{n \in \mathbb Z} f_n \, e^{i n \theta} .
\]
Observe that $f_0 =0$ when (H1) is fulfilled and $f_{\pm 1} =0$ when (H2) is
fulfilled. For the time being, let us assume that both (H1) and (H2) are satisfied.
Then, the (bounded) harmonic extension of $f$ is given explicitly by
\[
F (s, \theta) = \sum_{ | n | \geq 2 } e^{- |n| \, s} \, f_n \, e^{in \theta } .
\]
Since
\[
| f_n|\leq \| f \|_{L^\infty (S^1)}æ,
\]
we get the pointwise estimate
\[
| F(s,\theta) | \leq 2 \, \| f \|_{L^\infty
(S^1)} \, \sum_{n \geq 2} e^{-n s} \leq 2\, \| f \|_{L^\infty (S^1)} \, \frac{e^{-2s}}{1
- e^{-s}} ,
\]
which implies that
\[
\sup_{[1,\infty) \times S^1} e^{2s} \, | F(s,\theta) | \leq C \, \| f \|_{L^\infty (S^1)} .
\]
Increasing the value of $C >0$ if this is necessary, we can use the maximum
principle in the annular region $[0,1] \times S^1$ to get
\[
\sup_{[0,\infty) \times S^1} e^{2s} \, | F(s,\theta) | \leq C \, \| f \|_{L^\infty (S^1)} .
\]
The estimates for the derivatives of $F$ then follow from classical elliptic estimates
since Schauder's estimates can be applied on each annulus $[s, s+1]æ\times S^1$,
for all $s \geq 0$. This already completes the proof of the result when both (H1)
and (H2) are fulfilled. When only (H1) holds, one has to take into account the
function $f_{\pm 1} \, e^{- s} \, e^{ \pm i \theta }$ which accounts for the slower
decay of $F$ as $e^{-s}$.
\end{proof}

\section{The Delaunay nodo\"{\i}ds}

\subsection{Parameterization and notations}

The Delaunay nodo\"{\i}d ${\mathfrak D}_\tau$
is a surface of revolution which can be parameterized by
\begin{equation}
X_\tau (s, \theta) : = \left( \phi_\tau (s) \, \cos \theta , \phi_\tau (s) \, \sin \theta ,
\psi_\tau (s) \right) ,
\label{eq:1.0}
\end{equation}
where $(s, \theta) \in \mathbb R \times S^1$. Here, the functions $\phi_\tau$ and $
\psi_\tau$ depend on the real parameter $\tau >0$ but, unless this is necessary,
we shall not make this apparent in the notation anymore. The function $\phi$ is
chosen to be the unique smooth, periodic, non-constant solution of
\begin{equation}
{\dot \phi}^2 + \left( \phi^2- \tau\right)^2 = \phi^2 ,
\label{eq:1.1}
\end{equation}
which takes its minimal value at $s=0$ (we agree that $\cdot$ denotes
differentiation with respect to the parameter $s$) and the function $\psi$ is
obtained by integration of
\begin{equation}
\dot \psi = \phi^2- \tau ,
\label{eq:1.2}
\end{equation}
with initial condition $\psi (0) =0$. Observe that $\phi$ is a smooth solution of
\begin{equation}
\ddot \phi + 2 \, \phi \, \left( \phi^2- \tau \right) = \phi .
\label{eq:1.00}
\end{equation}
Since $\phi^2 -\tau$ changes sign, the function $\psi$ is not monotone and closer
inspection of the solutions shows that ${\mathfrak D}_\tau$ is actually not
embedded. The Delaunay nodo\"{\i}ds also arise as the surface of revolution
whose generating curve is a {\em roulette} of a hyperbola and we refer to \cite{Eel}
for a description of this construction. The quantity $\frac 1 4 \, \tau$ is sometimes
referred to as the vertical flux of the Delaunay surface $\mathfrak D_\tau$ (see
Definition 3.1 in \cite{Ros}).

We define
\[
\underline \tau : = \frac{\sqrt{1+4\tau} -1}{2} \qquad \mbox{and} \qquad \overline
\tau : = \frac{\sqrt{1+4 \tau} +1}{2} ,
\]
which, thanks to (\ref{eq:1.1}), are respectively the minimum and maximum values
of $\phi$. As already mentioned, the function $\phi$ is periodic. We agree that $s_
\tau$ denotes one {\em half} of the fundamental period of $\phi$. Using
(\ref{eq:1.1}), we can write
\begin{equation}
s_\tau = \int_{\underline \tau}^{\overline \tau} \frac{d\zeta}{\sqrt{ \zeta^2 - (\zeta^2-
\tau)^2}} .
\label{eq:1/2period}
\end{equation}

In the above parameterization, the induced metric on ${\mathfrak D}_\tau$ is given
by
\[
g_\tau : = \phi^2 \, (ds^2 + d\theta^2) ,
\]
and it is easy to check that the second fundamental form on ${\mathfrak D}_\tau$ is
given by
\[
h_\tau : = (\phi^2 + \tau) \, ds^2 + (\phi^2 - \tau ) \, d\theta^2 ,
\]
when the unit normal vector field is chosen to be
\[
N_\tau : = \frac{1}{\phi} \, \left( (\tau - \phi^2) \, \cos \theta , (\tau - \phi^2) \, \sin
\theta , \dot \phi \right) .
\]
Finally, the tensor $k_\tau $ is given by
\[
k_\tau : = \left( \phi +\frac{ \tau}{\phi} \right)^2 \, ds^2 + \left( \phi - \frac{\tau}{\phi}
\right)^2 \, d\theta^2 .
\]
In particular, the formula for the induced metric and the second fundamental form
implies that the mean curvature of this surface is constant equal to
\[
H : = \frac{1}{2} \, {\rm tr}^{g} \, h =1 .
\]
In these coordinates, it follows at once from the expression of $g_\tau$ and $k_\tau
$, that the Jacobi operator about ${\mathfrak D}_\tau$ is given by
\[
J_\tau : = \frac{1}{2 \, \phi^2} \, \left( \partial_s^2 + \partial_\theta^2 + 2 \,
\left( \phi^2 + \frac{\tau^2}{\phi^2}\right) \right).
\]

\subsection{Structure and refined asymptotics}

The structure of the Delaunay surfaces ${\mathfrak D}_\tau$ is well understood
and it is known that, as the parameter $\tau$ tends to $0$, ${\mathfrak D}_\tau$
converges to the union of infinitely many spheres of radius $1$ which are
arranged periodically along the vertical axis. To get a better grasp on the structure
of ${\mathfrak D}_\tau$ as $\tau$ tends to $0$, we have the following results which
were already used in many constructions of constant mean curvature surfaces by
gluing \cite{Maz-Pac-1}, \cite{Maz-Pac-Pol} and \cite{Maz-Pac-Pol-2}. For the sake
of completeness we give here independent proofs of these results.

First, we have the~:
\begin{lemma}
As $\tau$ tends to $0$ the following holds~:

\begin{itemize}
\item[(i)] The sequence of functions $\phi_\tau (\cdot + s_\tau)$ converges
uniformly on compacts of $\mathbb R$ to the function $s \longmapsto (\cosh
s)^{-1}$.\\

\item[(ii)] The sequence of functions $\psi_\tau (\cdot + s_\tau) - \psi_\tau (s_\tau)$
converges uniformly on compacts of $\mathbb R$ to the function $s \longmapsto
\tanh s$.
\end{itemize}
\end{lemma}
\begin{proof}
It is easy to check that $\phi_\tau (\cdot + s_\tau)$ is even and that $\phi_\tau (s_
\tau)= \overline \tau$ converges to $1$ as $\tau$ tends to $0$ (this follows from the
fact that the function $\phi_\tau$ achieves its maximal value when $s =s_\tau$).
Passing to the limit in (\ref{eq:1.1}), we conclude that the sequence of functions $
\phi_\tau$ converges uniformly on compacts of $\mathbb R$ to a function $\phi_0$
which is a solution of
\[
{\ddot \phi}_0 + 2 \, \phi^3_0 = \phi_0 .
\]
Moreover, the function $\phi_0$ is even and is equal to $1$ when $s=0$.
Therefore, necessarily $\phi_0 (s) = (\cosh s )^{-1}$. Next, one can pass to the limit
in (\ref{eq:1.2}) to prove that the sequence $\psi_\tau (\cdot + s_\tau) - \psi_\tau (s_
\tau)$ converges to a function $\psi_0$ which is a solution of
\[
\dot \psi_0 = \phi^2_0 ,
\]
and satisfies $\psi_0(0)=0$. We find that $\psi_0 (s) = \tanh s$ and this completes
the proof of the result.
\end{proof}

Now, we investigate the behavior of ${\mathfrak D}_\tau$ close to the origin in $
\mathbb R^3$. It turns out that the sequence of rescaled surfaces $\frac{1}{\tau} \,
{\mathfrak D}_\tau$ converges on compacts of $\mathbb R^3$ to a cateno\"{\i}d
whose axis is the vertical axis. This is the content of the~:
\begin{lemma}
As $\tau$ tends to $0$, the following holds~:
\begin{itemize}
\item[(i)] The sequence of functions $\frac{1}{\tau} \, \phi_\tau $ converges
uniformly on compacts of $\mathbb R$ to the function $s \longmapsto \cosh s$. \\

\item[(ii)] The sequence of functions $\frac{1}{\tau} \, \psi_\tau $ converges
uniformly on compacts of $\mathbb R$ to the function $s \longmapsto - s$.
\end{itemize}
\label{le:2.2}
\end{lemma}
\begin{proof}
It is easy to check that $\phi_\tau$ is even and that $\frac{1}{\tau} \, \phi_\tau (0) =
\underline \tau / \tau$ converges to $1$ as $\tau$ tends to $0$ (this follows from
the fact that the function $\phi_\tau$ achieves its minimum value when $s =0$).
Passing to the limit in (\ref{eq:1.1}), we conclude that $\frac{1}{\tau} \, \phi_\tau$
converges uniformly on compacts of $\mathbb R$ to a function $\phi_0$ which is a
solution of
\[
{\ddot \phi}_0 = \phi_0 .
\]
Moreover, $\phi_0$ is even and is equal to $1$ when $s=0$. Therefore, $\phi_0
(s) = \cosh s$. Next, one can pass to the limit in (\ref{eq:1.2}) to prove that the
sequence $\frac{1}{\tau} \, \psi_\tau $ converges to a function $\psi_0$ which is a
solution of
\[
\dot \psi_0 = - 1 ,
\]
and satisfies $\psi_0(0)=0$. Therefore, we conclude that $\psi_0 (s) = - s$ as
desired.
\end{proof}

Geometrically, these results show that, as $\tau$ tends to $0$, the Delaunay
surface $\mathfrak D_\tau$ is close to infinitely many spheres of radius $1$, which
are arranged along the vertical axis (and are slightly overlapping) and each
sphere is connected to its two neighbors by small rescaled cateno\"{\i}ds.

We will need a refined and more quantitative version of Lemma~\ref{le:2.2}.
Observe that $\dot \phi < 0$ in $(- s_\tau, 0)$ and hence $\phi$ is a
diffeomorphism from $(- s_\tau , 0)$ into $(\underline \tau, \overline \tau)$.
We can define the change of variables
\[
r = \phi_\tau (s) ,
\]
to express $s \in (- s_\tau , 0)$ as a function of $r \in (\underline \tau, \overline \tau)
$ and write
\[
X_\tau (s, \theta) = \left( r\, \cos \theta , r \, \sin \theta , u_\tau (r) \right),
\]
for some function $u_\tau$ defined in an annulus of $\mathbb R^2$.
Geometrically, this means that the image of $(- s_\tau, 0) \times S^1$ by $X_\tau$
is a vertical graph for some function $u_\tau$ which is defined over the annulus
\[
\{ x \in \mathbb R^2 \, : \, \underline \tau < |x| < \overline \tau \} .
\]
The next result gives a precise expansion of the function $u_\tau$ as $\tau$ tends
to $0$.
\begin{proposition}
As $\tau$ tends to $0$,
\[
u_\tau (r) = \frac{\tau}{\sqrt{1+ 2\tau} } \, \log \left( \frac{2 \,r}{\tau } \right) +{\mathcal
O}_{\mathring {\mathcal C}^\infty} \left( \frac{\tau^3}{r^2} \right) + {\mathcal
O}_{\mathring {\mathcal C}^\infty} (r^2) ,
\]
for $r \in (2 \, \underline \tau, \frac 1 2 \, \overline \tau )$, uniformly as $\tau$ tends
to $0$.
\label{pr:4.1}
\end{proposition}
The notation $f_1 = {\mathcal O}_{\mathring {\mathcal C}^\infty} (f_2)$ means that
the function $f_1$ and all its derivatives with respect to the vector fields $r \,
\partial_r$ and $\partial_\theta$ are bounded by a constant (depending on the
order of derivation) times the (positive) function $f_2$.

\begin{proof} By definition $\underline \tau$ is the minimal value of $\phi$. Hence,
we can write
\[
\phi (s) = \underline\tau \, \cosh (w(s)) ,
\]
for some function $w$ which vanishes at $s =0$. Plugging this into (\ref{eq:1.1})
we find that the function $w$ is a solution of
\[
\dot w^2 = 1 + 2 \tau -\underline \tau^2 \, (1 + \cosh^2 w) .
\]
As long as $|w(s) - s | \leq 1$, we can estimate,
\[
w (s) = \sqrt{1+ 2 \tau} \, s + {\mathcal O} (\tau^2 \, \cosh^2 s) .
\]
In particular, we conclude {\it a posteriori} that $|w(s) - s | \leq 1$ holds, and hence
that the above estimate is justified, provided $|s| \leq - \log \tau - c$, for some
constant $c >0$ independent of $\tau \in (0, 1)$. In the range of study, we are
entitled to consider the change of variable
\[
r = \underline \tau \, \cosh w (s) ,
\]
and express $s < 0$ as a function of $r$. We find
\begin{equation}
\sqrt{1+2\tau} \, s = - \log \left( \frac{2\, r}{\underline \tau} \right) + \mathcal O
\left( \frac{\tau^2}{r^2} \right) + \mathcal O (r^2) .
\label{eq:200}
\end{equation}
Finally, using (\ref{eq:1.2}), we can write
\[
\dot \psi = - \tau + \underline \tau^2 \, \cosh^2 w .
\]
Integrating over $s$ we get
\[
\psi (s) = - \tau s + \mathcal O (\tau^2 \, \cosh^2 s) ,
\]
and the result follows directly from (\ref{eq:200}) together with the fact that, by
definition
\[
u_\tau (r) = \psi (s) .
\]
Similar estimates can be obtained for the derivatives of $u_\tau$.
\end{proof}

A close inspection of the proof of Proposition~\ref{pr:4.1} also yields the~:
\begin{lemma}
As $\tau$ tends to $0$, half of the fundamental period of the function $\phi_\tau
$ can be expanded as
\[
s_\tau = - \log \tau + \mathcal O (1) ,
\]
and there exists a constant $C >1$ such that
\[
\frac{\tau}{C} \cosh s \leq \phi_\tau \leq C \, \tau \cosh s ,
\]
when $ s \in (- s_\tau, s_\tau)$; this estimate being uniform as $\tau$ tends to $0$.
\label{le:4.33}
\end{lemma}
\begin{proof}
The asymptotic of the half period of $\phi$ can also be derived from the formula
(\ref{eq:1/2period}). The estimate for $\phi$ follows from the proof of Proposition~
\ref{pr:4.1}.
\end{proof}

\subsection{Analysis of the Jacobi operator}

We analyze the mapping properties of the Jacobi operator about the Delaunay
surface ${\mathfrak D}_\tau$, paying special attention to what happens when $
\tau$ tends to $0$. This analysis is very close to the one available in
\cite{Maz-Pac-1} or \cite{Hau-Pac}. Again, we give here a self contained proof
which is adapted to the nonlinear argument we will use in the subsequent
sections.

We first analyze the behavior, as $\tau$ tends to $0$, of the potential which
appears in the expression of $J_\tau$. To this aim we assume that we are given
for each $\tau >0$ a real number $t_\tau \in \mathbb R$ and we define
\[
\xi_\tau : = \left( \phi_\tau^2 + \frac{\tau^2}{\phi_\tau^2}\right) (\cdot - t_\tau) .
\]
We have the~:
\begin{lemma}
As $\tau$ tends to $0$, a subsequence of the sequence of functions $\xi_\tau$
converges uniformly on compacts of $\mathbb R$ either to $s \longmapsto (\cosh
(s - s_0))^{-2}$, for some $s_0 \in \mathbb R$, or to $0$.
\label{le:2.4}
\end{lemma}
\begin{proof}
We define
\[
\zeta_\tau : = \left( \phi_\tau + \frac{\tau}{\phi_\tau}\right) (\cdot - t_\tau) ,
\]
and we observe that, using (\ref{eq:1.1}) we find that $\zeta_\tau$ is a solution of
\begin{equation}
\dot \zeta_\tau^2 = (æ\zeta_\tau^2 - 2 \, \tau) \, (1+ 4 \, \tau - \zeta_\tau^2) ,
\label{eq:2.1}
\end{equation}
and (\ref{eq:1.00}) also implies that
\begin{equation}
\ddot \zeta_\tau = \zeta_\tau \, (1+ 6 \, \tau - 2 \, \zeta_\tau^2) .
\label{eq:2.11}
\end{equation}
Now, we can estimate
\[
\zeta_\tau^2 = \left( \phi - \frac{\tau}{\phi}\right)^2 + 4 \tau \leq 1+ 4 \, \tau ,
\]
where we have used (\ref{eq:1.1}) which provides the estimate $(\phi^2-\tau)^2
\leq \phi^2$. This implies that $\zeta_\tau$ and its derivatives remain bounded as
$\tau$ tends to $0$. We can then let $\tau$ tend to $0$ and pass to the limit in
(\ref{eq:2.1}) and (\ref{eq:2.11}) to get that, as $\tau$ tends to $0$, the sequence $
\zeta_\tau$ converges on compacts to a solution of the equation $\ddot \zeta =
\zeta \, (1- 2 \, \zeta^2) $ which satisfies $\dot \zeta^2 = æ\zeta^2 \, (1- \zeta^2)$.
Hence $\zeta$ is either $0$ or a translation of $(\cosh s)^{-2}$. The result
then follows from the identity $\xi_\tau = \zeta_\tau^2 - 2 \, \tau$.
\end{proof}

We denote by $\pm \delta_j (\tau)$, for $j \in \mathbb N$, the indicial roots of the
operator $J_\tau$. Recall that the indicial roots $\pm \delta_j$ correspond to the
indicial roots of
\[
J_{\tau, j} : = \frac{1}{2 \, \phi^2} \, \left( \partial_t^2 -j^2+ \tau^2 \, \left( \phi^2 +
\frac{\tau^2}{\phi^2}\right) \right) ,
\]
which appears in the Fourier decomposition of the operator $J_\tau$ in the $\theta
$ variable. By definition, the indicial roots of $J_{\tau, j}$ characterize the
exponential growth or decay rate at infinity of the solutions of the homogeneous
problem $J_{\tau, j} \, w =0$. In general, it is a very hard problem to determine the
exact value of the indicial roots of an operator, in the present case, taking
advantage of the geometric nature of the problem, we can prove the following~:
\begin{proposition}
For all $\tau >0$,
\[
\delta_0 (\tau) = \delta_1 (\tau)=0 .
\]
Furthermore, for $j \geq 2$
\[
\delta_j(\tau) \geq \sqrt{j^2-2- 4\tau} ,
\]
provided $\tau < \sqrt{j^2-2}$.
\label{pr:deltaj}
\end{proposition}
\begin{proof}
The fact that $\delta_0 (\tau) =0$ follows from the observation that the function $
\frac{\dot \phi}{\phi}$ is periodic and solves
\[
J_{\tau, 0} \, \left( \frac{\dot \phi}{\phi} \right) =0 .
\]
This either follows from direct computation or can be derived from the fact that
\[
\Phi^+_0 : = \frac{\dot \phi}{\phi} ,
\]
is the Jacobi field associated to vertical translation (see \S 2.2). Since the function
$\phi$ is periodic, the homogeneous problem $J_{\tau, 0} \, w =0$ has a bounded
solution and this implies that $\delta_0( \tau ) =0$.

The fact that $\delta_1(\tau) =0$ follows from the observation that the function $
\phi - \frac{\tau}{\phi}$ is periodic and solves
\[
J_{\tau, 1} \, \left( \phi - \frac{\tau}{\phi}\right) =0 .
\]
Again, this either follows from direct computation or can be derived from the fact
that
\[
\Phi^+_1 : = \left( \phi - \frac{\tau}{\phi}\right) \, \cos \theta \qquad \mbox{and}
\qquad \Phi^+_{-1} : = \left( \phi - \frac{\tau}{\phi}\right) \, \sin \theta ,
\]
are the Jacobi fields associated to translations perpendicular to the axis of the
Delaunay surface.

The estimate from below for the other indicial roots follows from the fact that,
according to (\ref{eq:1.1})
\[
2 \, \left( \phi^2 + \frac{\tau^2}{\phi^2}\right) = 2 \, \left( \phi - \frac{\tau}{\phi}
\right)^2 + 4 \, \tau \leq 2 +æ4 \, \tau .
\]
This in particular implies that the potential in $\partial_s^2 -j^2 + 2 \, \left( \phi^2 +
\frac{\tau^2}{\phi^2}\right)$ can be estimated from below by $\bar \delta_j^2$,
where
\[
\bar \delta_j : = \sqrt{j^2-2- 4\, \tau } .
\]
The result then follows from the maximum principle and standard ODE arguments
since the function $s \longmapsto e^{\bar \delta_j s}$ can be used as a barrier to
prove the existence of two positive solutions of $J_{\tau, j} \, w =0$ which are
defined on $(0, \infty)$, one of which being bounded from above by $e^{-\bar
\delta_j s}$ and the other one being bounded from below by $e^{\bar \delta_j s}$.
In particular, this implies that $\delta_j \geq \bar \delta_j$ and hence, this
completes the proof of the result.
\end{proof}

For all $\delta\in \mathbb R$, we define the operator
\[
\begin{array}{rcccllll}
L_\delta : & e^{\delta s} \, \mathcal C^{2, \alpha} (\mathbb R \times S^1) &
\longrightarrow & e^{\delta s} \, \mathcal C^{0, \alpha} (\mathbb R \times S^1)
\\[3mm]
& w & \longmapsto & \phi^2 \, J_{\tau} \, w ,
\end{array}
\]
where the norms in the function spaces $\mathcal C^{k, \alpha} (\mathbb R \times
S^1)$ are computed with respect to the cylindrical metric $g_{cyl}$. Observe that
the Jacobi operator has been multiplied by the conformal factor $\phi^2$ and
hence
\[
\phi^2 \, J_\tau = \frac 1 2 \, \left( \partial_s^2 + \partial_\theta^2 + 2 \, \left( \phi^2 +
\frac{\tau^2}{\phi^2}\right) \right) .
\]
Also, this operator depends on the parameter $\tau$. We now study the mapping
properties of $ \phi^2 \, J_\tau$ as the parameter $\tau$ tends to $0$. The following
result selects a range of weights for which the norm of the the solution of $L_\delta
\, w = f$ is controlled, uniformly as $\tau$ tends to $0$.
\begin{proposition}
Assume that $|\delta | >1$, $\delta \notin \mathbb Z$ is fixed. Then, there exist $
\tau_\delta >0$ and $C >0$, only depending on $\delta$, such that, for all $\tau \in
(0, \tau_\delta)$ and for all $w \in e^{\delta s} \, \mathcal C^{2, \alpha} (\mathbb R
\times S^1)$, we have
\[
\|æe^{- \delta s} \, w \|_{\mathcal C^{2, \alpha} (\mathbb R \times S^1) }æ\leq C\, \| e^{-
\delta s} \, L_\delta \, w \|_{\mathcal C^{0, \alpha} (\mathbb R \times S^1) } .
\]
\label{pr:linres10}
\end{proposition}
\begin{proof}
Observe that, thanks to Schauder's elliptic estimates, it is enough to prove that
\[
\|æe^{- \delta s} \, w \|_{L^\infty (\mathbb R \times S^1) }æ\leq C\, \| e^{-\delta s} \, \phi_
\tau^2 \, J_\tau \, w \|_{L^\infty (\mathbb R \times S^1) } ,
\]
provided $\tau$ is close enough to $0$. The proof of this estimate is by
contradiction. Assume that, for some sequence $\tau_n$ tending to $0$ there
exists a sequence of functions $w_n$ such that
\[
\|æe^{- \delta s} \, w_n \|_{L^\infty (\mathbb R \times S^1) } =1 \qquad \mbox{and}
\qquad \lim_{j\rightarrow \infty} \| e^{-\delta s} \, \phi_{\tau_n}^2 \, J_{\tau_n} \,
w_n \|_{L^\infty (\mathbb R \times S^1) } = 0 .
\]
Pick a point $s_n \in \mathbb R$ such that $\|æe^{- \delta s_n} \, w_n (s_n , \cdot)
\|_{L^\infty (S^1) } \geq 1/2$ and define the rescaled sequence
\[
\bar w_n ( s , \theta) : = e^{-\delta \, s_n } \, w ( s + s_n , \theta) .
\]
We still have
\[
\|æe^{- \delta s} \, \bar w_n \|_{L^\infty (\mathbb R \times S^1) } =1 \qquad
\mbox{and} \qquad \lim_{j\rightarrow \infty} \| e^{-\delta s} \, \bar L_{n} \bar w_n
\|_{L^\infty (\mathbb R \times S^1) } = 0 ,
\]
where by definition $\bar L_{n}$ is defined by
\[
\bar L_{n} : = \partial_s^2 + \partial_\theta^2 + 2 \, \left( \phi_{\tau_n}^2 +
\frac{\tau_n^2}{\phi_{\tau_n}^2}\right) (\cdot + s_n) .
\]

Elliptic estimates and Ascoli-Arzela's theorem allows one to extract some
subsequence and pass to the limit as $n$ tends to $\infty$ to get a function
$w_\infty$ which is a nontrivial solution to either
\begin{equation}
(\partial_s^2 + \partial_\theta^2) \, w_\infty = 0 ,
\label{eq:2.2}
\end{equation}
or
\begin{equation}
\left( \partial_s^2 + \partial_\theta^2 + \frac{2}{\cosh^2 (\cdot + s_*)} \right) \, w_\infty
= 0 ,
\label{eq:2.3}
\end{equation}
according to the different cases described in Lemma~\ref{le:2.4}. To simplify the
notations, we assume that $s_* =0$, straightforward modifications are needed to
handle the general case. Observe that we also have
\[
\| e^{-\delta s} \, w_\infty \|_{L^\infty (\mathbb R \times S^1)} \leq 1,
\]
and $\| w_\infty (0, \cdot)\|_{L^\infty (S^1)} \geq 1/2$.

We decompose $w_\infty$ as
\[
w_\infty (s, \theta) = \sum_{j \in \mathbb Z} w^{(j)} (s) \, e^{ij\theta} .
\]
It is easy to prove that for any solution of (\ref{eq:2.2}), $w^{(j)}$ is a linear
combination of $e^{\pm js}$ and none is bounded by a constant times $e^{\delta s}
$ unless $\delta$ is an integer (which we have assumed not to be the case).

Similarly, if $w_\infty$ is a solution of (\ref{eq:2.3}), we find that $w^{(j)}$ is a
solution of
\begin{equation}
\left(\partial_s^2 - j^2 + \frac{2}{æ\cosh^2 s} \right) \, w^{(j)} =0 ,
\label{eq:2.3q}
\end{equation}
and is either asymptotic to $e^{js}$ or to $e^{-js}$ at $\pm \infty$. Inspection of the
behavior of (\ref{eq:2.3q}) at infinity then implies that there is no such solution which is bounded
by a constant times $e^{\delta s}$ if $|j| < |\delta | $. When $|j| > |\delta |$,
inspection of the behavior of (\ref{eq:2.3q}) at infinity implies that any such solution is necessarily
bounded and the maximum principle then implies that this solution is identically
$0$ (observe that in this case $j^2 >2$ since $|j| > |\delta | > 1$ and hence the
potential in (\ref{eq:2.3q}) is negative).

When $j=0$, all solutions of
\[
\left(\partial_s^2 + \frac{2}{æ\cosh^2 s} \right) \, w^{(0)} =0 ,
\]
are linear combination of the functions $\tanh s$ and $1-s \, \tanh s $ and none is
bounded by a constant times $e^{\delta s}$ unless $\delta =0$ (which is not the
case).

Finally, and this is the reason why we had to choose $|\delta | >1$, when $j=1$, all
solutions of
\[
\left(\partial_s^2 - 1 + \frac{2}{æ\cosh^2 s} \right) \, w^{(1)} =0 ,
\]
are linear combination of the functions $(\cosh s)^{-1}$ and $s \, ( \cosh s)^{-1}+
\sinh s$ and none is bounded by a constant times $e^{\delta s}$ unless $|\delta|
\leq 1$ (which is contrary to our assumption). Again we have reached a
contradiction. Having reached a contradiction in all cases, the proof of the
Proposition is complete.
\end{proof}

Thanks to the previous result, we can now describe the mapping properties of $
\phi^2 \, J_\tau$ for the range of weights $\delta$ of interest for our problem.
\begin{proposition}
Assume that $|\delta | >1$, $\delta \notin \mathbb Z$ is fixed. Then, there exist $
\tau_\delta >0$ and $C >0$, only depending on $\delta$, such that, for all $\tau \in
(0, \tau_\delta)$, the operator $L_\delta$ is an isomorphism the norm of whose
inverse is bounded independently of $\tau$.
\label{pr:linres100}
\end{proposition}
\begin{proof}
Injectivity follows at once from Proposition~\ref{pr:linres10}. As far as surjectivity is
concerned, we give here a simple self contained proof in the case where $\delta
\in (1, \sqrt 2)$ (or $\delta \in (-\sqrt 2, -1)$). We will then sketch a general proof.

To fix the ideas, let us assume that $\delta \in (1, \sqrt 2)$. We first assume that $f
\in \mathcal C^{0, \alpha} (\mathbb R\times S^1)$ has
compact support and we decompose it as
\[
f (s, \theta) = f_0(s) + f_{\pm 1} (s) \, e^{\pm i \, \theta} + \bar f (s, \theta),
\]
where, by definition,
\[
\bar f : = \sum_{j \neq 0, \pm 1}æf_j (s) \, e^{ij\theta}
\]

Observe that, if we restrict our attention to functions $\bar w$ whose Fourier
decomposition in the $\theta$ variable is of the form
\[
\bar w (s, \theta) = \sum_{j \neq 0, \pm 1} w_j(s) \, e^{ij\theta} ,
\]
we have
\begin{equation}
\int_{\mathbb R\times S^1} \left( |\partial_s \bar w|^2 + |\partial_\theta \bar w|^2 - 2 \,
\left( \phi^2+ \frac{\tau^2}{\phi^2}\right) \, \bar w^2 \right) \, ds \, d\theta \geq (2 -4
\tau) \, \int_{\mathbb R \times S^1} \bar w^2 \, ds \, d\theta .
\label{coerc}
\end{equation}
This follows at once from the estimate of the potential involved in the expression of
$J_\tau$ which has been obtained in the proof of Proposition~\ref{pr:deltaj},
namely
\begin{equation}
2 \, \left( \phi^2+ \frac{\tau^2}{\phi^2}\right) \leq 2 + 4 \, \tau,
\label{esza}
\end{equation}
together with the fact that
\[
\int_{\mathbb R\times S^1} |\partial_\theta \bar w|^2 \, ds \, d\theta \geq 4 \,
\int_{\mathbb R \times S^1} \bar w^2 \, ds \, d\theta .
\]

Thus, if we assume that $\sqrt{2} \, \tau < 1$, this inequality implies that we can
solve
\[
\phi^2 \, J_\tau \, \bar w = \bar f,
\]
in $H^1(\mathbb R\times S^1)$. Elliptic estimates then imply that $\bar w \in
\mathcal C^{2, \alpha} (\mathbb R \times S^1)$. Finally, the solvability of
\[
\phi^2 \, J_\tau \, (w_j \, e^{ij\theta} ) = f_j \, e^{ij\theta},
\]
for $j=0, \pm1$, follows easily from integration of the associated second order
ordinary differential equation starting from $-\infty$, hence $w_j \equiv 0$ when $s
$ is close to $-\infty$. Obviously, the function
\[
w : = w_0 + w_{\pm 1}\, e^{\pm i\theta} + \bar w ,
\]
is a solution of the equation $\phi^2 \, J_\tau \, w = f$.

We claim that, provided $\tau$ is chosen small enough, $w \in e^{\delta s} \,
\mathcal C^{2, \alpha} (\mathbb R \times S^1)$. Assuming that the claim is already
proven, the result of Proposition~\ref{pr:linres10} applies and we get
\[
\|æe^{-\delta s} \, w \|_{\mathcal C^{2, \alpha} (\mathbb R \times S^1)} \leq C \, \| e^{-
\delta s} \, f\|_{\mathcal C^{0, \alpha} (\mathbb R \times S^1)}
\]
for any function $f$ having compact support. The general result, when $f$ does not
necessarily have compact support, follows from a standard exhaustion argument.
We choose a sequence of functions $f^{(n)} \in \mathcal C^{0, \alpha} (\mathbb R
\times S^1)$ having compact support converging on compacts to a given function
$f\in e^{\delta s} \, \mathcal C^{0, \alpha} (\mathbb R \times S^1)$. Moreover,
without loss of generality, we can assume that
\[
\|æe^{-\delta s} \, f^{(n)} \|_{\mathcal C^{0, \alpha} (\mathbb R \times S^1)} \leq C \, \|
e^{- \delta s} \, f \|_{\mathcal C^{0, \alpha} (\mathbb R \times S^1)},
\]
for some constant $C >0$ independent of $n\geq 0$. Thanks to the above, we
have a sequence of solutions of $\phi^2 \, J_\tau \, w^{(n)} = f^{(n)}$ satisfying
\[
\|æe^{-\delta s} \, w^{(n)} \|_{\mathcal C^{2, \alpha} (\mathbb R \times S^1)} \leq C \, \|
e^{- \delta s} \, f \|_{\mathcal C^{0, \alpha} (\mathbb R \times S^1)} .
\]
Extracting some subsequence and passing to the limit, one gets the existence of
$w \in e^{\delta s} \, \mathcal C^{0, \alpha} (\mathbb R \times S^1)$, a solution of $
\phi^2 \, J_\tau \, w = f$ satisfying
\[
\|æe^{-\delta s} \, w \|_{\mathcal C^{0, \alpha} (\mathbb R \times S^1)} \leq C \, \| e^{-
\delta s} \, f \|_{\mathcal C^{0, \alpha} (\mathbb R \times S^1)}.
\]
The result then follows from Schauder's estimates.

It remains to prove the claim. We keep the notations introduced above. We first
prove that $\bar w$ tends to $0$ exponentially fast at infinity. Indeed, away from
the support of $\bar f$, we can multiply the equation $ \phi^2 \, J_\tau \, \bar w =
\bar f$ by $\bar w$ and integrate over $S^1$ to get
\[
\frac{1}{2} \, \frac{d^2}{ds^2} \left( \int_{S^1} \bar w^2 \, d\theta \right) = \int_{S^1}
\left( |\partial_s \bar w|^2 + |æ\partial_\theta \bar w|^2 - \left( \phi^2 +
\frac{\tau^2}{\phi^2}\right) \, \bar w^2 \right) \, d\theta \, .
\]
But
\[
\int_{S^1} |æ\partial_\theta \bar w|^2 \, d\theta \geq 2 \, \int_{S^1} \bar w^2 \, d\theta ,
\]
and we conclude from (\ref{esza}) that
\[
\frac{d^2}{ds^2} \, \left( \int_{S^1} \bar w^2 \, d\theta \right) \geq 4 \, ( 1- 2\, \tau) \,
\int_{S^1} \bar w^2 \, d\theta \, .
\]
Since we have assumed that $\delta \in (1,\sqrt{2})$, we can assume that $\tau >0$
is small enough so that
\[
\delta^2 \leq 2 \, ( 1- 2 \tau) ,
\]
and using the fact that that $\bar w$ is bounded, we conclude that there exists
$C >0$, such that
\[
\int_{S^1} \bar w^2 \, d\theta \leq C \, (\cosh s)^{-2\delta} .
\]
This shows that $\bar w \in (\cosh s)^{-\delta} \, L^2 (\mathbb R \times S^1)$ and, by
elliptic regularity, this implies that $\bar w \in (\cosh s)^{-\delta} \,\mathcal C^{2,
\alpha} (\mathbb R \times S^1)$.

Finally, it remains to check that the functions $w_0$ and $w_{\pm 1}$ are at most
linearly growing at $+\infty$. This follows at once from the fact that, for $s$ large
enough, these functions are solutions of second order homogeneous ordinary
differential equations
\[
\left(\partial_s^2 - j^2 + 2 \, \left( \phi^2 + \frac{\tau^2}{\phi^2} \right) \right) w_j =0 .
\]
For $j=0,1$, this ordinary differential equation whose potential is periodic, has one
solution which is periodic (see the proof of Proposition~\ref{pr:deltaj}) and standard
result imply that the other linearly independent solution of this ordinary differential
equation is at most linearly growing (see Appendix 1). In particular, $w_j \in
e^{\delta s} \, \mathcal C^{2, \alpha} (\mathbb R)$ and this completes the proof of
the claim.

We briefly explain how the proof of the general result can be obtained. The idea is
to solve the equation $\phi^2 \, J_\tau \, \bar w_{s_0} =\bar f$ in $[-s_0, s_0] \times
S^1$ with $0$ boundary conditions, this can be done using the coercivity
inequality (\ref{coerc}). Then, the proof of Proposition~\ref{pr:linres10} can be
adapted to prove that
\[
\|æe^{-\delta s} \, \bar w_{s_0} \|_{\mathcal C^{2, \alpha} ([-s_0, s_0] \times S^1)} \leq
C \, \| e^{- \delta s} \, \bar f \|_{\mathcal C^{0, \alpha} (\mathbb R \times S^1)} ,
\]
for some constant $C>0$ independent of $s_0 >1$ (observe that we use the fact
that the Fourier decomposition of the function $\bar w$ in the $\theta$ variable
does not have any component over $1$ and $e^{\pm i\theta}$). It then remains to
pass to the limit in the sequence $w_{s_0}$ as $s_0$ tends to $\infty$ to prove the
existence of $\bar w$ solution of $\phi^2 \, J_\tau \, \bar w =\bar f$ in $\mathbb R
\times S^1$ which satisfies the correct estimate.
\end{proof}

Using similar arguments, one can give a direct proof of the following general
result, which will not be needed in this paper ~:
\begin{theorem}
Assume that $\delta \neq \pm \delta_j(\tau)$, for all $j \in \mathbb N$, then $L_
\delta$ is an isomorphism.
\label{th:4.1}
\end{theorem}
The proof of this result follows from the general theory developed in \cite{Pac} (see
Theorem 10.2.1 on page 61 and Proposition 12.2.1 on page 81) or in \cite{Mel},
\cite{Maz}, \ldots

In what follows we will restrict our attention to functions which are invariant under
some symmetries. More precisely, we will assume that the functions are invariant
under the action on $S^1$ of the dihedral group ${\rm Dih}^{(2)}_{m+1}$ of
isometries of $\mathbb R^2$ which leave a regular polygon with $m+1$ sides fixed.
The operator associated to $\phi^2 \, J_\tau$, acting on the weighted space of
functions which are invariant under these symmetries, will be denoted by $L_
\delta^\sharp$. This time $L_\delta^\sharp$ is an isomorphism provided $\delta
\neq \pm \delta_j$ for all $j \inæ\mathbb Z$ for which there exist eigenfunctions of $
\partial_\theta^2$ which are invariant under the action of ${\rm Dih}^{(2)}_{m+1}$,
namely, $j \notin m\, \mathbb Z$. Observe that, when $j=1$, there are no such
eigenfunctions and hence, working equivariantly allows us to extend the range in
which the weight parameter $\delta$ can be chosen.

Close inspection of the previous proof shows that the range in which the weight $
\delta$ can be chosen so that the inverse of $L^\sharp_\delta$ remains bounded
as $\tau$ tends to $0$ can be enlarged if we work equivariantly. Even though we
will not use it, we state here the corresponding result for the sake of completeness.
\begin{proposition}
Assume that $\delta \notin m \, \mathbb Z$ is fixed. Then, there exist $\tau_\delta
>0$ and $C >0$, only depending on $\delta$, such that, for all $\tau \in (0, \tau_
\delta)$ and for all $w \in e^{\delta s} \, \mathcal C^{2, \alpha} (\mathbb R \times
S^1)$ which is invariant under the action of ${\rm Dih}^{(2)}_{m+1}$, we have
\[
\|æe^{- \delta s} \, w \|_{\mathcal C^{2, \alpha} (\mathbb R \times S^1) }æ\leq C\, \| e^{-
\delta s} \, L^\sharp_\delta \, w \|_{\mathcal C^{0, \alpha} (\mathbb R \times S^1) } .
\]
\label{pr:5.10}
\end{proposition}

\subsection{The mean curvature of normal graphs over ${\mathfrak D}_\tau$}

In this section, we investigate the mean curvature of a surface which is a normal
graph over ${\mathfrak D}_\tau$. Given a smooth function $w$ (small enough)
defined on ${\mathfrak D}_\tau$, we consider the surface parameterized by
\[
\tilde X (s,\theta) = X_\tau (s, \theta) + w(s, \theta) \, N_\tau (s, \theta) .
\]
We have the following technical result~:
\begin{lemma}
The mean curvature of the surface parameterized by $\tilde X$ is given by
\[
H(w) = 1 + J_\tau w + \frac{1}{\phi} \, Q_\tau \left( \frac{w}{\phi} \right) ,
\]
where the second order differential nonlinear operator $Q_\tau$ depends on
$\tau$ and satisfies
\[
\begin{array}{rlllll}
\|æQ_\tau ( v_2) - Q_\tau (v_1) \|_{\mathcal C^{0, \alpha} ([s, s+1] \times S^1)} \leq c
\, \left( \|æv_1 \|_{\mathcal C^{2, \alpha} ([s, s+1] \times S^1)} + \|æv_2 \|_{\mathcal
C^{2, \alpha} ([s, s+1] \times S^1)} \right) \\[3mm]
\times \, \|v_2 - v_1 \|_{\mathcal C^{2, \alpha} ([s, s+1] \times S^1)} ,
\end{array}
\]
for some constant $c >0$ independent of $s$ and $\tau \in (0,1)$, for all functions
$v_1$, $v_2$ satisfying $\|æv_i \|_{\mathcal C^{2, \alpha} ([s, s+1 ] \times S^1)}
\leq 1$.
\label{le:mcffncf1}
\end{lemma}
\begin{proof}
This follows at once from Proposition~\ref{pr:mcng} together with the fact that the
functions $\phi$, $\frac \tau \phi$ and $\frac{\dot \phi}{\phi}$ as well as the
derivatives of these functions are uniformly bounded as $\tau$ tends to $0$.
Indeed, we have
\[
\begin{array}{rllll}
g_w & = & g- 2\, w \, h + w^2 \, k \\[3mm]
& = &\displaystyle \phi^2 \, \left( \left( 1 - \left(\phi + \frac{\tau}{\phi} \right) \,
\frac{w}{\phi} \right)^2 \, ds^2 + \left( 1 - \left(\phi - \frac{\tau}{\phi} \right) \,
\frac{w}{\phi} \right)^2 \, d\theta^2 \right) .
\end{array}
\]
Hence $\phi^{-2} \, g_w$ has coefficients which are bounded functions of
$\frac{w}{\phi}$. Similarly, the tensor $\phi^{-1} (h- w\, k)$ also has coefficients
which are bounded functions of $\frac{w}{\phi}$. Using this, it is straightforward to
check that the nonlinear terms in $H(w)$ are a function of $\frac{\partial_s^k
\partial_\theta^\ell w}{\phi}$, for $k+\ell=0,1,2$ with coefficients bounded by $
\frac{1}{\phi}$. Finally, observe that
\[
\frac{\partial_s w}{\phi} = \partial_s \left( \frac{w}{\phi} \right) + \frac{\dot \phi}{\phi} \,
\frac{w}{\phi} ,
\]
and hence, any expressions of the form $ \frac{\partial_s w}{\phi}$ can also be
expressed as a linear combination (with coefficients bounded uniformly as $\tau$
tends to $0$) of the function $\frac{w}{\phi}$ and its derivatives. We leave the
details to the reader.
\end{proof}

\subsection{A first fixed point argument}

We assume that we are given $\tau >0$. We define $\bar s \in (- s_{\tau} , 0)$ by the
identity
\[
\phi_{\tau} (\bar s ) = \tau^{3/4} .
\]
Observe that $\bar s$ depends on $\tau$, even though we have chosen not to
make this apparent in the notation. Moreover, it follows from the proof of
Proposition~\ref{pr:4.1} that
\[
\bar s = \frac{1}{4}æ\log \tau +æ\mathcal O (1) ,
\]
as $\tau$ tends to $0$. We define the truncated nodo\"{\i}d $\mathfrak D_{\tau}^+$
to be the image of $ [\bar s, +\infty) \times S^1$ by $X_\tau$. Observe that this
surface has a boundary and, thanks to Proposition~\ref{pr:4.1}, close to this
boundary, it can be parameterized as the vertical graph of the function
\[
x \longmapsto \tau \, \log \left( \frac{2 \, |x|}{\tau} \right) + \mathcal
O_{\mathring {\mathcal C}^\infty}(\tau^{3/2}),
\]
over the annulus $\overline D(0, \tau^{3/4}) - D(0, 2\tau)$. Moreover, $\mathfrak
D_{\tau}^+$ has one end in the upper half space.

In this section, we apply the implicit function theorem (or to be more precise, a fixed
point argument for a contraction mapping) to produce an infinite dimensional family
of constant mean curvature surfaces which are close to $\mathfrak D_{\tau}^+$,
have one boundary which can be described using a function $f : S^1
\longrightarrow \mathbb R$.
\begin{proposition}
\label{pr:4.4}
Assume that we are given $\kappa >0$ large enough (the value of $\kappa$ will be
fixed later on). Then, for all $\tau >0$ small enough and for all functions $f$
invariant under the action of ${\rm Dih}^{(2)}_{m+1}$, satisfying (H1) (observe that
(H2) is automatically satisfied) and
\begin{equation}
\|æf\|_{\mathcal C^{2, \alpha} (S^1)} \leq \kappa \, \tau^{3/2} ,
\label{eq:estimf}
\end{equation}
there exists a constant mean curvature surface $\mathfrak D_{\tau ,f}^+$ with mean
curvature equal to $1$, which is a graph over $\mathfrak D_{\tau}^+$, has one
Delaunay end asymptotic to the end of $\mathfrak D_\tau^+$ and one boundary.
When $f=0$, $\mathfrak D_{\tau , 0 } =\mathfrak D_{\tau}^+$ and, close to its
boundary, the surface $\mathfrak D_{\tau ,f}$ is a {\em vertical graph} over the
annulus
\[
\left\{ x \in \mathbb R^2 \, : \, \frac 1 2 \, \tau^{3/4} \leq |x| \leq \tau^{3/4} \right\} ,
\]
for the function $x \longmapsto U^\upharpoonright_{\tau ,f}( \tau^{-3/4} \, x) $ which
can be expanded as follows
\begin{equation}
U^\upharpoonright_{\tau ,f} (x) = \tau \, \log \left( \frac{2}{\tau^{1/4}} \right) + \tau \,
\log |x| - W^{\rm ins}_f (x) + \bar U^\upharpoonright_{\tau , f} (x) ,
\label{eq:5.199}
\end{equation}
where we recall that $W^{\rm ins}_f$ denotes the bounded harmonic extension of
$f$ in the punctured unit disc and where
\begin{equation}
\| \bar U^\upharpoonright_{\tau, 0} \|_{\mathcal C^{2, \alpha} ( \overline D(0, 1) - D
(0, 1/2) )} \leq C \, \tau^{3/2} .
\label{eq:est-1}
\end{equation}
Moreover, the nonlinear operator
\[
\mathcal C^{2, \alpha} (S^1) \ni f \longmapsto \bar U^\upharpoonright_{\tau ,f} \in
\mathcal C^{2, \alpha} ( \overline D(0, 1) - D (0, 1/2) ),
\]
is Lipschitz and, given $\delta \in (-2,-1)$, we have
\begin{equation}
\| \bar U^\upharpoonright_{\tau ,f'} - \bar U^\upharpoonright_{\tau , f} \|_{\mathcal
C^{2, \alpha} ( \overline D(0, 1) - D (0, 1/2) )} \leq C \, \tau^{(2+\delta)/4} \, \| f' -f
\|_{\mathcal C^{2, \alpha}( S^1)} ,
\label{eq:est-2}
\end{equation}
for some constant $C > 0$, independent of $\kappa$, $\tau$ and $f, f'$. Finally, $
\mathfrak D_{\tau ,f}^+$ is invariant under the action of the dihedral group ${\rm
Dih}^{(2)}_{m+1}$.
\end{proposition}

Before we proceed with the proof of the Proposition, one comment is due. Observe
that we have chosen to describe the surface near its boundary as the graph of the
function $x \longmapsto U^\upharpoonright_{\tau ,f} ( \tau^{-3/4} \, x)$ and as a
consequence the function $U^\upharpoonright_{\tau ,f}$ is defined over the
annulus $ \overline D(0, 1) - D (0, 1/2)$. Alternatively, we could have chosen not to
scale the coordinates and have a function defined over the annulus $\overline D(0,
\tau^{3/4}) - D (0, \tau^{3/4}/2)$, which would be more natural. However, with this
latter choice, we would have to consider in (\ref{eq:est-1}) and (\ref{eq:est-2}),
function spaces where partial derivatives are taken with respect to the vector fields
$r\, \partial_r$ and $\partial_\theta$ to evaluate the norm of these functions, while
with the former choice, the H\"older spaces are the usual ones.

\begin{proof}
The proof of this result is fairly technical but by now standard. To begin with, in the
annular region which is the image of $( \bar s-2, \bar s+2) \times S^1$ by $X_{\tau}
$, we modify the unit vector field $N_{\tau}$ into $\bar N_{\tau}$ in such a way that
$\bar N_{\tau}$ is equal to $- e_3$, the downward pointing unit normal vector field
on the image of $( \bar s - 1, \bar s + 1) \times S^1$ by $X_{\tau}$. Using
Proposition~\ref{pr:2.2}, direct estimates imply that the expression of the mean
curvature given in Lemma~\ref{le:mcffncf1} has to be altered into
\[
\bar H(w) = 1 + J_{\tau} w + \frac{1}{\phi^2} \, \ell_{\tau} \, w + \frac{1}{\phi} \, \bar
Q_{\tau} \left( \frac{w}{\phi} \right) ,
\]
where $\bar Q_{\tau}$ enjoys properties which are similar to the properties enjoyed
by $Q_{\tau}$ and where $\ell_{\tau}$ is a linear second order partial differential
operator in $\partial_s$ and $\partial_\theta$ whose coefficients are smooth, have
support in $[\bar s-2, \bar s+2] \times S^1$ and are bounded (in $\mathcal C^\infty$
topology) by a constant (independent of $\tau$) times $\tau^{1/2}$. This estimate
comes from the fact that
\[
N \cdot (- e_3) = 1 + \mathcal O (\tau^{1/2}) ,
\]
on the image of $[\bar s-2, \bar s+2] \times S^1$ by $X_\tau$.

We assume that we are given a function $f \in \mathcal C^{2, \alpha } (S^1)$
satisfying both (H1), (H2) and (\ref{eq:estimf}) and we denote by $F$ the harmonic
extension of $f$ in $(\bar s , \infty) \times S^1$.

Given these data, we would like to solve the nonlinear equation
\begin{equation}
\phi^2 \, J_{\tau} ( F+ w) + \ell_{\tau} \, (F+w) + \phi \, \bar Q_{\tau} \left(
\frac{F+w}{\phi} \right) = 0 ,
\label{eq:mce1}
\end{equation}
in $(\bar s , \infty) \times S^1$. Provided $w$ is small enough and decays
exponentially at infinity, this will then provide constant mean curvature
surfaces which are close to a half nodo\"{\i}d ${\mathfrak D}_{\tau}^+$.

We choose
\[
\mathcal E_\tau : \mathcal C^{0, \alpha} ([\bar s ,\infty) \times S^1)æ\longrightarrow
\mathcal C^{0, \alpha} (\mathbb R \times S^1) ,
\]
an extension operator such that
\[
\left\{ \begin{array}{rlllll}
\mathcal E_\tau (\psi) & = & \psi \qquad & \mbox{in}æ\qquad [\bar s , \infty) \times S^1
\\[3mm]
\mathcal E_\tau (\psi) & = & 0 \qquad & \mbox{in} \qquad (-\infty, \bar s - 1] \times
S^1 ,
\end{array}
\right.
\]
and
\[
\| \mathcal E_\tau (\psi) \|_{\mathcal C^{0, \alpha} ([\bar s - 1, \bar s +1] \times S^1) }
\leq C \, \| \psi \|_{\mathcal C^{0, \alpha} ([\bar s , \bar s + 1] \times S^1) } .
\]
We rewrite (\ref{eq:mce1}) as
\begin{equation}
\phi^2 \, J_{\tau} w = - \mathcal E_\tau \, \left( \phi^2 \, J_{\tau} (F+w) + \ell_{\tau}
\, F + \phi \, \bar Q_{\tau} \left( \frac{F+w}{\phi} \right) \right) ,
\label{eq:mce2}
\end{equation}
where, this time, the function $w$ is defined on all $\mathbb R \times S^1$ (to be
more precise, one should say that, on the right hand side, we consider the
restriction of $w$ to $[\bar s , \infty) \times S^1$).

The following estimates follow easily if one uses the fact that
\[
\frac{C}{\tau} \, \cosh s \leq \phi \leq C \, \tau \, \cosh s \qquad \mbox{in} \qquad
(- s_\tau , s_\tau) ,
\]
for some $C >1$, and also that $\phi$ is periodic of period $2 \, s_\tau$. We
assume that $\delta \in (-2,-1)$ is fixed. It is easy to check that there exists a
constant $c >0$ (independent of $\kappa$) and a constant $c_\kappa >0$
(depending on $\kappa$) such that
\[
\left\|æ e^{-\delta s} \, \mathcal E_\tau \left( \left( \phi^2 + \frac{\tau^2}{\phi^2}\right) \, F
\right) \right\|_{\mathcal C^{0, \alpha} (\mathbb R\times S^1)} \leq c \, \tau^{1/2} \, \| f
\|_{\mathcal C^{2, \alpha} (S^1)} \, ,
\]
\[
\left\|æ e^{-\delta s} \, \mathcal E_\tau \left(\ell_\tau \, (F + w) \right) \right\|_{\mathcal
C^{0, \alpha} (\mathbb R\times S^1)} \leq c \, \tau^{1/2} \, \left( \| e^{-\delta s}
\, w \|_{\mathcal C^{2, \alpha}_\delta (\mathbb R \times S^1)} + \tau^{- \delta/4} \, \| f
\|_{\mathcal C^{2, \alpha} (S^1)} \right) ,
\]
and we also have
\[
\begin{array}{llll}
\displaystyle \left\|æ e^{-\delta s} \, \mathcal E_\tau \left( \phi \, Q_\tau \left( \frac{w'+
F '}{\phi} \right) - \phi \, Q_\tau \left( \frac{w+ F}{\phi} \right) \right) \right\|_{\mathcal
C^{0, \alpha} (\mathbb R\times S^1)} \\[3mm]
\qquad \qquad \leq c_\kappa \, \left( \tau^{3/4} \, \| e^{-\delta s} \, (w' -w) \|_{\mathcal
C^{2, \alpha}_\delta (\mathbb R \times S^1)} + \tau^{(3-\delta)/4} \, \| f' -f\|_{\mathcal
C^{2, \alpha} (S^1)} \right) ,
\end{array}
\]
provided $w$ and $w'$ satisfy
\[
\| e^{-\delta s} \, w \|_{\mathcal C^{2, \alpha} (\mathbb R \times S^1)}æ+ \| e^{-\delta
s} \, w' \|_{\mathcal C^{2, \alpha} (\mathbb R \times S^1)}æ\leq C_\kappa \, \tau^2,
\]
for some fixed constant $C_\kappa >0$. Here $F$ and $F'$ are respectively the
harmonic extensions of the boundary data $f$ and $f'$.

At this stage, we make use of the result of Theorem~\ref{th:4.1} (or more precisely
its equivariant version) to rephrase (\ref{eq:mce2}) as a fixed point problem in
$e^{\delta s} \, \mathcal C^{2, \alpha} (\mathbb R \times S^1)$. The estimates we
have just derived are precisely enough to solve this nonlinear problem using a
fixed point agument for contraction mappings in the ball of radius $C_\kappa \,
\tau^2$ in $e^{\delta s} \, \mathcal C^{2, \alpha} (\mathbb R \times S^1)$, where
$C_\kappa$ is a constant which is fixed large enough. Therefore, for all $\tau >0$
small enough, we find a solution $w$ of (\ref{eq:mce2}) satisfying
\[
\| e^{-\delta s} \, w\|_{\mathcal C^{2, \alpha} (\mathbb R \times S^1)}æ\leq C_\kappa
\, \tau^2 .
\]
In addition, it follows from the above estimates that
\[
\| e^{-\delta s} \, (w' - w ) \|_{ \mathcal C^{2, \alpha} (\mathbb R \times S^1)}æ\leq C_
\kappa \, \tau^{1/2} \, \| f' -f\|_{\mathcal C^{2, \alpha} (S^1)} ,
\]
where $w$ (resp. $w'$) is the solution associated to $f$ (resp. $f'$).

To complete the result, it is enough to change coordinates $r = \phi (s)$ in the
range where $\frac{1}{2} \, \tau^{3/4} \leq r \leq 2 \, \tau^{3/4}$ and $| s - \bar s | \leq
1$. There is no real difficulty in deriving the estimates (\ref{eq:est-1}) and
(\ref{eq:est-2}) which follow from Proposition~\ref{pr:linres10} and the estimate for
$w$. One should be aware that there is some subtlety here, since one should pay
attention to the fact that if we change variables $r = \phi (s)$ for $\frac{1}{2} \,
\tau^{3/4} \leq r \leq 2 \, \tau^{3/4}$ and $|s - \bar s|æ\leq 1$, then $F(s)$ is not equal
to $W^{\rm ins}_f (\phi(s))$ (because $s$ does not correspond to the cylindrical
coordinate $r = e^{t}$ in $\mathbb R^2- \{0\}$) ! In fact we have
\[
F(s) = W^{\rm ins}_f (e^{\bar s - s}) ,
\]
and $\tau^{-3/4} \, r = \phi (s) / \phi (\bar s)$ and is not equal to $e^{\bar s -s}$.
Nevertheless, using the expansion of $\phi$ derived it is easy to check that
\[
\| F (\bar s - \log \phi(s) + \log \phi(\bar s) ) - F ( s )\|_{\mathcal C^{2, \alpha} ( [\bar
s , \bar s+ 2] \times S^1)} \leq c \, \tau^{1/2} \, \| f\|_{\mathcal C^{2, \alpha} (S^1)} ,
\]
for some constant $c >0$ independent of $\tau$.
\end{proof}

\section{The cateno\"{\i}d}

\subsection{Parameterization and notations}

We recall some well known facts about cateno\"{\i}ds in Euclidean space. The
normalized cateno\"{\i}d $\mathfrak C$ is the minimal surface of revolution which is
parameterized by
\[
Y_0(s,\theta) : = (\cosh s \, \cos \theta, \cosh s \, \sin \theta, s) ,
\]
where $(s, \theta) \in \mathbb R \times S^1$. The induced metric on $\mathfrak C$
is given by
\[
g_0 : = (\cosh s)^2 \, (ds^2 + d\theta^2) ,
\]
and it is easy to check that the second fundamental form is given by
\[
h_0 : = ds^2 - d\theta^2 ,
\]
when the unit normal vector field is chosen to be
\[
N_0 : = \frac{1}{\cosh s} \, \left( \cos \theta , \sin \theta , - \sinh s \right) .
\]
In particular, the formula for the induced metric and the second fundamental form
implies that the mean curvature of the surface $\mathfrak C$ is constant equal to $0$.

In the above defined coordinates, the Jacobi operator about the catenoid is given
by
\[
J_0 : = \frac{1}{2 \, \cosh^2 s} \, \left( \partial_s^2 + \partial_\theta^2 +
\frac{2}{\cosh^2 s} \right).
\]

\subsection{Refined asymptotics}

We are interested in the parameterization of the catenoid as a (multivalued)
vertical graph over the horizontal plane. We consider for example the lower part of
the cateno\"{\i}d as the graph, over the complement of the unit disc in the horizontal
plane for the function $u_0$. Namely, $u_0$ is the negative function defined by
\[
u_0 (\cosh s ) = s ,
\]
for all $s \leq 0$. It is easy to check that
\[
s = - \log (2 \, r) + {\mathcal O}_{\mathring{\mathcal C}^\infty} (r^{-2}) ,
\]
and hence the lower end of the cateno\"{\i}d can also be parameterized as a
vertical graph over $\mathbb R^2 - D(0,1)$ by
\[
(r, \theta ) \longmapsto \left( r \cos \theta, r \, \sin \theta, u_0 (r) \right) .
\]
With little work, one proves the~:
\begin{lemma}
The following expansion holds
\[
u_0 (r) = - \log (2 \, r) + {\mathcal O}_{\mathring{\mathcal C}^\infty} (r^{-2}) ,
\]
in $\mathbb R^2 - D(0,2)$.
\label{le:5.11}
\end{lemma}

\subsection{Mapping properties of the Jacobi operator about the cateno\"{\i}d}

The functional analysis of the Jacobi operator about the cateno\"{\i}d is well
understood and some results can be found for example in \cite{Maz-Pac-Pol}.
Again the indicial roots of $J_0$ caracterize the asymptotic behavior of the
solutions of the homogeneous problem $J_{0,j} \, w =0$ where
\[
J_{0, j} : = \frac{1}{2 \, \cosh^2 s} \, \left( \partial_s^2 - j^2 + \frac{2}{\cosh^2 s}
\right) ,
\]
It is easy to see that the indicial roots of $J_{0, j}$ are equal to $\pm j$ and hence
the indicial roots of $J_0$ are equal to $\pm j$, for $j \in \mathbb N$.

For all $\delta \in \mathbb R$, we define the operator
\[
\begin{array}{rcccccclllll}
\mathcal L_\delta : & (\cosh s)^\delta \, \mathcal C^{2, \alpha} ( \mathbb R\times
S^1) & \longrightarrow & (\cosh s)^\delta \, \mathcal C^{0, \alpha} (\mathbb R
\times S^1)\\[3mm]
& w & \longmapsto & (\cosh s)^2 \, J_0 \, w ,
\end{array}
\]
where, as usual, the norms in the function spaces $\mathcal C^{k, \alpha} (\mathbb
R \times S^1)$ are computed with respect to the cylindrical metric $g_{cyl}$.

Paralleling what we have proven in \S 4, we have the~:
\begin{proposition}
Assume that $\delta \in (-2,2)$. Then, there exists $C >0$, only depending on $
\delta$, such that, for all $\bar w \in (\cosh s)^\delta \, \mathcal C^{2, \alpha}
(\mathbb R \times S^1)$, we have
\[
\|æ (\cosh s)^{-\delta} \, \bar w \|_{\mathcal C^{2, \alpha} (\mathbb R \times
S^1) }æ\leq C\, \| (\cosh s)^{-\delta} \, \mathcal L_\delta \, \bar w \|_{\mathcal C^{0,
\alpha} (\mathbb R \times S^1) } ,
\]
provided
\begin{equation}
\int_{S^1} \bar w (s,\theta)\, d \theta = \int_{S^1} \bar w (s,\theta)\, e^{\pm i \theta} \,
d \theta = 0 ,
\label{oc}
\end{equation}
for all $s \in \mathbb R$.
\label{pr:linres10bis}
\end{proposition}
\begin{proof}
The proof is really parallel to the proof of Proposittion~\ref{pe:linres10} and is left to
the reader.
\end{proof}

The following result follows from the general theory developed in \cite{Pac} (see
Theorem 10.2.1 on page 61 and Proposition 12.2.1 on page 81) or \cite{McO},
\cite{Mel}, \cite{Maz}, \ldots For the sake of completeness we provide a
self-contained proof.
\begin{theorem}
Assume that $\delta \in (1, 2)$ then $\mathcal L_\delta$ is surjective and has a
$6$-dimensional kernel.
\label{th:5.1}
\end{theorem}
\begin{proof}
The proof follows the proof of Proposition~\ref{pr:linres100}.

Recall that the action of rigid motions and dilations provides many Jacobi fields.
For example,
\begin{equation}
J_0 \, \left( \tanh s \right) =0 \qquad \mbox{and} \qquad J_0 \, \left( 1 - s \tanh s
\right) =0,
\label{eq:jf1}
\end{equation}
which either follows from direct computation or from the fact that these are the
Jacobi fields associated to the group of vertical translations and the group of
dilations centered at the origin.

Similarly
\begin{equation}
J_0 \, \left( \frac{1}{\cosh s} \, e^{\pm i \theta} \right) =0 \qquad \mbox{and} \qquad
J_0 \, \left( \left( \sinh s + \frac{1}{\cosh s} \right) \, \, e^{\pm i \theta} \right) =0,
\label{eq:jf2}
\end{equation}
which again either follows from direct computation or from the fact that these are
the Jacobi fields associated to the group of horizontal translations and the group of
rotations about the vertical axis, centered at the origin.

This already shows that, when $\delta \in (1, 2)$, the kernel of $\mathcal L_\delta$
is at least $6$-dimensional.

We first assume that $f \in \mathcal C^{0, \alpha} (\mathbb R\times S^1)$ has
compact support and we decompose it as $f (s, \theta) = f_0(s) + f_{\pm 1} (s) \,
e^{\pm i \, \theta} + \bar f (s, \theta)$ where, by definition,
\[
\bar f : = \sum_{j \neq 0, \pm 1}æf_j (s) \, e^{ij\theta} .
\]

If we restrict our attention to functions $\bar w$ whose Fourier decomposition in the
$\theta$ variable is of the form
\[
\bar w (s, \theta) = \sum_{j \neq 0, \pm 1} w_j(s) \, e^{ij\theta} ,
\]
we have
\begin{equation}
\int_{\mathbb R\times S^1} \left( |\partial_s \bar w|^2 + |\partial_\theta \bar w|^2 -
\frac{2}{ \cosh^2 s} \, \bar w^2 \right) \, ds \, d\theta \geq 2 \, \int_{\mathbb R \times
S^1} \bar w^2 \, ds \, d\theta .
\label{coercbis}
\end{equation}
Therefore, we can solve
\[
(\cosh s)^2 \, J_0 \, \bar w = \bar f,
\]
in $H^1(\mathbb R\times S^1)$. Elliptic estimates then imply that $\bar w \in
\mathcal C^{2, \alpha} (\mathbb R \times S^1)$.

Obviously $\bar w \in (\cosh s)^\delta \, \mathcal C^{2, \alpha} (\mathbb R \times
S^1)$, since $\delta >0$. The result of Proposition~\ref{pr:linres10bis} applies and
we get
\[
\|æ (\cosh s)^{-\delta} \, \bar w \|_{\mathcal C^{2, \alpha} (\mathbb R \times S^1)} \leq
C \, \| (\cosh s)^{-\delta} \, f \|_{\mathcal C^{0, \alpha} (\mathbb R \times S^1)} ,
\]
for any function $f$ having compact support. The general result, when $f$ does not
necessarily have compact support, follows from a standard exhaustion argument.

Finally, the solvability of
\[
(\cosh s)^2 \, J_0 \, (w_j \, e^{ij\theta} ) = f_j \, e^{ij\theta},
\]
for $j=0, \pm1$, follows easily from integration of the associated second order
ordinary differential equation starting from $0$ (with initial data and initial velocity
equal to $0$). We have explicitly
\[
w_j = A_j^+ \, \int_0^s A_j^- (t) \, f_j(t) \, dt - A_j^- \, \int_0^s A_j^+ (t) \, f_j(t) \, dt ,
\]
where $A_j^\pm$ are the two independent solutions of
\[
\left( \partial_s^2 -j^2+ \frac{2}{\cosh^2 s} \right) \, A_j^\pm =0,
\]
which are given in (\ref{eq:jf1}) and (\ref{eq:jf2}) and are normalized so that their
Wronskian is equal to $1$. Direct estimates imply that
\[
\|æ (\cosh s)^{-\delta} \, w_j \|_{\mathcal C^{2, \alpha} (\mathbb R )} \leq C \, \| (\cosh
s)^{-\delta} \, f \|_{\mathcal C^{0, \alpha} (\mathbb R \times S^1)} ,
\]
provided $\delta > 1$ (more precisely, $\delta >0$ is needed to derive the estimate
for $w_0$ and $\delta >1$ is needed to derive the estimate for $w_{\pm1}$).
We set $w = w_0 + w_{\pm 1}\, e^{\pm i \theta} + \bar w$. This completes the proof
of the fact that the operator $\mathcal L_\delta$ is surjective when $\delta \in
(1,2)$. The fact that this operator, restricted to the space of functions satisfying the
orthogonality conditions (\ref{oc}) is injective follows from the result of Proposition~
\ref{pr:linres10bis}. Hence the kernel of $\mathcal L_\delta$ is $6$-dimensional.
\end{proof}

\subsection{The mean curvature of normal graphs over the cateno\"{\i}d}

We consider in this section the mean curvature of a surface which is a normal
graph over $\mathfrak C$. Hence, for some smooth (small enough) function $w$
defined on $\mathfrak C$, we consider the surface parameterized by
\[
Y (s,\theta) = Y_0 (s, \theta) + w (s, \theta) \, N_0 (s, \theta) .
\]
We have the following technical result~:
\begin{lemma}
The mean curvature of the surface parameterized by $Y$ is given by
\[
H(w) = J_0 w + \frac{1}{\cosh s} \, Q_0 \left( \frac{w}{\cosh s} \right) ,
\]
where the nonlinear second order differential operator $Q_0$ satisfies
\[
\begin{array}{rlllll}
\|æQ_0 ( v_2) - Q_0 (v_1) \|_{\mathcal C^{0, \alpha} ([s, s+1] \times S^1)} \leq c \,
\left( \|æv_1 \|_{\mathcal C^{2, \alpha} ([s, s+1] \times S^1)} + \|æv_2 \|_{\mathcal
C^{2, \alpha} ([s, s+1] \times S^1)} \right) \\[3mm]
\times \, \|v_2 - v_1 \|_{\mathcal C^{2, \alpha} ([s, s+1] \times S^1)} ,
\end{array}
\]
for some constant $c >0$ independent of $s$ and $\tau \in (0,1)$, for all functions
$v_1$, $v_2$ satisfying $\|æv_i \|_{\mathcal C^{2, \alpha} ([s, s+1] \times S^1)} \leq
1$.
\end{lemma}
\begin{proof}
This result is already proven in \cite{Maz-Pac-1}. In any case, a simple proof follows
easilly from Proposition~\ref{pr:mcng} together with the fact that
\[
g_w = \cosh^2 s \, \left( \left(1 - \frac{w}{\cosh^2 s}\right)^2 \, ds^2 + \left(1+
\frac{w}{\cosh^2 s}\right)^2 \, d\theta^2\right).
\]
We leave the details to the reader.
\end{proof}

\subsection{A second fixed point argument}

Assume that $\tau , \tilde \tau >0$ are chosen small enough and satisfy
\begin{equation}
\left| \tilde \tau - \frac{\tau}{m+1} \right| \leq \kappa \, \tau^{3/2} ,
\label{eq:tt}
\end{equation}
where the constant $\kappa >0$ is fixed large enough (the value of $\kappa$ will
be fixed in the last section of the paper). The rational for this estimate will also be
explained in the last section of the paper. We define $\tilde s> 0$ by
\[
\tilde \tau \, \cosh \tilde s= \tau^{3/4} .
\]
Observe that $\tilde s$ depends on both $\tau$ and $\tilde \tau$ even though we
have chosen not to make this apparent in the notation. It is easy to check that $\tilde
s = - \frac{1}{4}æ\log \tau + \mathcal O (1)$. We define the truncated cateno\"{\i}d $
\mathfrak C_{\tilde \tau}$ to be the image of $[-\tilde s , \tilde s ] \times S^1$ by $
\tilde \tau \, Y_0$ (to simplify the notations, we do not write the dependence of this
surface on the parameter $\tau$).

Building on the previous analysis, we prove the existence of {\em constant mean
curvature surfaces} which are close to the truncated cateno\"{\i}d $\mathfrak
C_{\tilde \tau}$ and have two boundaries which can be described by some
function $f : S^1 \longrightarrow \mathbb R$. We also require that the surfaces are
invariant under the action of the symmetry with respect to the horizontal plane.
More precisely, we have the following~:
\begin{proposition}
Assume we are given $\kappa >0$ large enough (the value of $\kappa$ will be
fixed later on). For all $\tau, \tilde \tau >0$ small enough satisfying (\ref{eq:tt}) and
for all functions $f$ invariant under the action of the ${\rm Dih}^{(2)}_{m+1}$,
satisfying (H1) (notice that (H2) is automatically satisfied) and
\begin{equation}
\|æf\|_{\mathcal C^{2, \alpha} (S^1)} \leq \kappa \, \tau^{3/2} ,
\label{eq:estimf}
\end{equation}
there exists a constant mean curvature surface $\mathfrak C_{\tilde \tau ,f}$ which
is close to $\mathfrak C_{\tilde \tau}$ and has two boundaries. The surface $
\mathfrak C_{\tilde \tau ,f}$ is invariant under the action of $\mathcal S_3$,
the symmetry with respect to the horizontal plane $x_3=0$, $\mathcal S_2$,
the symmetry with respect to the plane $x_2=0$, and, close to its lower
boundary, the surface $\mathfrak C_{\tilde \tau ,f}$ is a {\em vertical graph}
over the annulus
\[
\left\{ x \in \mathbb R^2 \, : \, \frac{1}{2} \, \tau^{3/4} \leq |x| \leq \tau^{3/4} \right\} ,
\]
for some function $x \longmapsto U^\downharpoonright_{\tilde \tau ,f} (\tau^{-3/4} \,
x)$ which can be expanded as follows
\begin{equation}
U^\downharpoonright_{\tilde \tau ,f} (x) = - \tilde \tau \, \log \left( \frac{2 \,
\tau^{3/4} }{\tilde \tau } \right) - \tilde \tau \, \log |x| + W^{\rm ins}_f (x) + \bar U^
\downharpoonright_{\tilde \tau, f} ( x ),
\label{eq:5.200}
\end{equation}
where we recall that $W^{\rm ins}_f$ denotes the bounded harmonic extension of
$f$ in the punctured unit disc and where
\begin{equation}
\| \bar U^\downharpoonright_{\tilde \tau, 0} \|_{\mathcal C^{2, \alpha} ( \overline D(0,
1) - D (0, 1/2) )} \leq C \, \tau^{3/2} .
\label{eq:est-1bis}
\end{equation}
Moreover, the nonlinear mapping
\[
\mathcal C^{2, \alpha} (S^1) \ni f \longmapsto \bar U^\downharpoonright_{\tilde
\tau ,f} \in \mathcal C^{2, \alpha} ( \overline D(0, 1) - D (0, 1/2) ),
\]
is Lipschitz and, given $\delta \in (1,2)$, we have
\begin{equation}
\| \bar U^\downharpoonright_{\tilde \tau ,f'} - \bar U^\downharpoonright_{\tilde \tau ,
f} \|_{\mathcal C^{2, \alpha} ( \overline D(0, 1) - D (0, 1/2) )} \leq C \, \tau^{(2-\delta)/
4} \, \| f' -f \|_{\mathcal C^{2, \alpha}( S^1)} ,
\label{eq:est-2bis}
\end{equation}
for some constant $C > 0$ independent of $\kappa , \tau, \tilde \tau$ and $f, f'$.
The function $ \bar U^\downharpoonright_{\tilde \tau ,f} $ depends continuously
on $\tilde \tau$.
\label{pr:5.1}
\end{proposition}
\begin{proof}
The proof of this result is very close to the proof of Proposition~\ref{pr:4.4} so we
shall only insist on the main differences.

Again, in the annular region which is the image of $(- \tilde s-2, - \tilde s+2) \times
S^1$ by $Y_{0}$, we modify the unit vector field $N_{0}$ into $\bar N_{0}$ in such
a way that $\bar N_{0}$ is equal to $e_3$ on the image of $(- \tilde s- 1, \tilde s+ 1)
\times S^1$ by $\tilde \tau \, Y_{0}$. We perform a similar modification on the upper
half of the cateno\"{\i}d, on the image of $(\tilde s-2, \tilde s+2) \times S^1$ by
$Y_{0}$, so that our construction remains invariant under the action of the
symmetry with respect to the horizontal plane. In this case, using
Proposition~\ref{pr:2.2}, one can check that the expression of the mean curvature
given in Lemma~\ref{le:mcffncf1} has to be altered into
\[
\bar H(w) = J_{0} w + \frac{1}{\cosh^2 s} \, \ell_{0} \, w + \frac{1}{\cosh s} \, \bar
Q_{0} \left( \frac{w}{\cosh s} \right) ,
\]
where $\bar Q_{0}$ enjoys properties which are similar to the properties enjoyed
by $Q_{0}$ and where $\ell_{0}$ is a linear second order partial differential
operator in $\partial_s$ and $\partial_\theta$ whose coefficients are smooth,
supported in $(- \tilde s-2, - \tilde s+2) \times S^1$ and in $(\tilde s-2, \tilde s+2)
\times S^1$ and are bounded (in the smooth topology) by a constant, independent
of $\tau$, times $\tau^{1/2}$.

We assume that $f$ is chosen to satisfy (H1), (H2) and
\[
\|æf \|_{\mathcal C^{2, \alpha} (S^1)} \leq \kappa \, \tau^{1/2}.
\]
Observe that the norm of the boundary data $f$ is bounded by a constant times $
\tau^{1/2}$ and not $\tau^{3/2}$. The reason being that we are going to perturb the
image of $[-\tilde s , \tilde s] \times S^1$ by $Y_0$ and then scale the surface we
obtain by a factor $\tilde \tau$ instead of perturbing $\mathfrak C_{\tilde \tau}$ which
is the image of $[-\tilde s , \tilde s] \times S^1$ by $\tilde Y_0$. Also this is the
reason why the equation we will solve is $\bar H(w) = \tilde \tau$ and not $\bar
H(w) =1$.

We denote by $F$ the harmonic extension of $f$ in $(-\infty , \tilde s) \times S^1$
and we set
\[
\tilde F (s, \theta) : = F(s, \theta )+ F(-s, \theta)
\]
which is well defined in $[- \tilde s, \tilde s] \times S^1$. One should be aware that
the boundary data of $F$ is not exactly equal to $f$ but the difference between $F$
and $f$ on the boundary tends to $0$ as $\tau$ tends to $0$. More precisely we
have
\[
\| F - \tilde F \|_{\mathcal C^{2, \alpha} ([\tilde s-1, \tilde s ] \times S^1)} \leq C \, \tau \,
\| f \|_{\mathcal C^{2, \alpha} (S^1)} .
\]

We would like to solve the equation
\begin{equation}
(\cosh s)^2 \, J_0 ( \tilde F+ w ) + \ell_{0} \, (\tilde F +w) + \cosh s \, Q_0
\left( \frac{ \tilde F+ w}{\cosh s} \right) = (\cosh s)^2 \, \tilde \tau \, .
\label{eq:mce22}
\end{equation}
in $(-\tilde s, \tilde s) \times S^1$. This will provide constant mean curvature
surfaces with mean curvature equal to $\tilde \tau$, which are close to the truncated
cateno\"{\i}d. Again, the solvability of this nonlinear problem will follow from a fixed
point theorem for a contraction mapping.

We choose
\[
\bar {\mathcal E}_\tau : \mathcal C^{0, \alpha} ([- \tilde s, \tilde s] \times
S^1)æ\longrightarrow \mathcal C^{0, \alpha} (\mathbb R \times S^1)
\]
an extension operator such that
\[
\left\{
\begin{array}{rllll}
\bar{\mathcal E}_\tau (\psi) & = & \psi \qquad & \mbox{in}æ\qquad [- \tilde s , \tilde s ]
\times S^1 \\[3mm]
\bar{\mathcal E}_\tau (\psi) & = & 0 \qquad & \mbox{in} \qquad ((- \infty , - \tilde s -
1] \cup [\tilde s+1, \infty)) \times S^1 ,
\end{array}
\right.
\]
and
\[
\| \bar{\mathcal E}_\tau (\psi) \|_{\mathcal C^{0, \alpha} ([ - \tilde s - 1, - \tilde s + 1]
\times S^1)} \leq C \, \| \psi \|_{\mathcal C^{0, \alpha} ([ - \tilde s, - \tilde s+1] \times
S^1) }.
\]
and
\[
\| \bar{\mathcal E}_\tau (\psi) \|_{\mathcal C^{0, \alpha} ([ \tilde s-1, \tilde s+1] \times
S^1) } \leq C \, \| \psi \|_{\mathcal C^{0, \alpha} ([ \tilde s-1, \bar s'] \times S^1) } .
\]

We rewrite (\ref{eq:mce22}) as
\begin{equation}
(\cosh s)^2 \, J_{0} w = \bar {\mathcal E}_\tau \, \left( (\cosh s)^2 \, ( \tilde \tau -
J_0 \, \tilde F ) - \ell_{0} \, (\tilde F+ w) - \cosh s \, Q_0 \left( \frac{ \tilde F+ w}{\cosh
s} \right) \right) .
\label{eq:mce2}
\end{equation}
Again, on the right hand side it is understood that we consider the image by $\bar
{\mathcal E}_\tau$ of the restriction of the functions to $[ - \tilde s, \tilde s] \times
S^1$.

We assume that $\delta \in (1,2)$ is fixed. It is easy to check that there exists a
constant $c >0$ (independent of $\kappa$) and a constant $c_\kappa >0$
(depending on $\kappa$) such that
\[
\left\|æ(\cosh s)^{-\delta} \, \bar{\mathcal E}_\tau \left(\cosh^2 s \, \tilde \tau \right) \right
\|_{\mathcal C^{2, \alpha} (\mathbb R\times S^1) } \leq c \, \tau^{(2 +\delta)/4} ,
\]
\[
\left\|æ(\cosh s)^{-\delta} \, \bar{\mathcal E}_\tau \left(\frac{\tilde F}{\cosh^2 s} \right)
\right\|_{\mathcal C^{2, \alpha} (\mathbb R\times S^1) } \leq c \, \tau^{1/2} \, \|æ
f\|_{\mathcal C^{2, \alpha} (S^1)} ,
\]
\[
\begin{array}{llll}
\left\|æ(\cosh s)^{-\delta} \, \bar{\mathcal E}_\tau \left( \ell_{0} \, (\tilde F+ w) \right)
\right\|_{\mathcal C^{2, \alpha} (\mathbb R\times S^1) }\\[3mm]
\qquad \leq c \, \tau^{1/2} \, \left( \|æf\|_{\mathcal C^{2, \alpha} (S^1)} + \tau^{\delta/
4} \, \| (\cosh s)^{-\delta} \, w\|_{\mathcal C^{2, \alpha} (\mathbb R \times S^1)} ,
\right) ,
\end{array}
\]
and
\[
\begin{array}{llll}
\displaystyle \left\|æ(\cosh s)^{-\delta} \, \bar{\mathcal E}_\tau \left( \cosh s \, Q
\left( \frac{w'+ \tilde F'}{\cosh s} \right) - \cosh s \, Q \left( \frac{w+ \tilde F}{\cosh s}
\right) \right) \right\|_{\mathcal C^{2, \alpha} (\mathbb R\times S^1)} \\[3mm]
\qquad \leq c_\kappa \, \left( \tau^{3/4} \, \| (\cosh s)^{-\delta} \, (w' -w) \|_{\mathcal
C^{2, \alpha} (\mathbb R \times S^1)} + \tau^{(1+ \delta)/4} \, \| f' -f\|_{\mathcal C^{2,
\alpha} (S^1)} \right) ,
\end{array}
\]
provided $w$ and $w'$ satisfy
\[
\| (\cosh s)^{-\delta} \, w \|_{\mathcal C^{2, \alpha} (\mathbb R \times S^1)}æ + \|
(\cosh s)^{-\delta} \, w' \|_{\mathcal C^{2, \alpha} (\mathbb R \times S^1)}æ\leq C \,
\tau^{(2 + \delta) /4} ,
\]
for some fixed constant $C >0$ (independent of $\kappa, \tau$ and $f$). Here $
\tilde F$ and $\tilde F'$ are respectively the harmonic extensions of the boundary
data $f$ and $f'$.

Now, we make use of the result of Theorem~\ref{th:5.1} to rephrase the problem as
a fixed point problem and the previous estimates are precisely enough to solve this
nonlinear problem using a fixed point agument for contraction mappings in the ball
of radius $C \, \tau^{(2 + \delta )/4}$ in $(\cosh s)^{\delta} \, \mathcal C^{2, \alpha}
(\mathbb R \times S^1)$, where $C >0$ is fixed large enough independent of $
\kappa$ provided $\tau$ is small enough. Then, for all $\tau >0$ small enough, we
find that there exists a constant $C>0$ (independent of $\kappa$) $C_\kappa >0$
(depending on $\kappa$) such that, for all functions $f$ satisfying (H1), (H2) and
(\ref{eq:estimf}), there exists a unique solution $w$ of (\ref{eq:mce2}) satisfying
\[
\| (\cosh s)^{\delta} \, w\|_{\mathcal C^{2, \alpha} (\mathbb R \times S^1)}æ\leq C \,
\tau^{(2 + \delta) /4} .
\]
In addition, we have the estimate
\[
\| e^{-\delta s} \, (w' - w ) \|_{ \mathcal C^{2, \alpha} (\mathbb R \times S^1)}æ\leq C_
\kappa \, \tau^{1/2} \, \| f' -f\|_{\mathcal C^{2, \alpha} (S^1)} ,
\]
where $w$ (resp. $w'$) is the solution associated to $f$ (resp. $f'$).

To complete the result, we simply shrink the surface we have obtained by a factor $
\tilde \tau$ to get a surface whose mean curvature is constant and equal to $1$.
The description of this surface close to its boundaries follows from the arguments
already used in the proof of Proposition~\ref{pr:4.4}. Observe that, the solution of
(\ref{eq:mce2}) is obtained through a fixed point theorem for contraction mappings,
and it is classical to check that the solution we obtain depends continuously on the
parameters of the construction. In particular, the constant mean curvature surface
we obtain depends continuously on $\tilde \tau$ (in fact one can also prove
that the surface depends smoothly on $\tilde \tau$, but we shall not use this
property).

To prove that, near its lower boundary, the surface we have obtained is a vertical
graph for some function which enjoys the decomposition (\ref{eq:5.200}), we make
use of the expansion in Lemma~\ref{le:5.11}) and we follow the steps of the
construction. Notice that $\bar U^\downharpoonright_{\tilde \tau,f}$ collects many
remainders : the one coming from the expansion in Lemma~\ref{le:5.11}, the
difference between $F$ and $\tilde F$, the function $w$ solution of the fixed point
problem and also the change of coordinates which takes into account that the
variable $s$ does not correspond to the cylindrical coordinates in $\mathbb R^2 - \{0\}
$.
\end{proof}

\section{The unit sphere}

\subsection{Notations and definitions}

We denote by $x_1,x_1,x_3$ the coordinates in $\mathbb R^3$. We agree that $
{\mathcal S}_j$ denotes the symmetry with respect to the $x_j = 0$ plane, and, for
all $m \in \mathbb N$ and that ${\mathcal R}_{m+1}$ denotes the rotation of angle
$\frac{2\pi}{m+1}$ about the $x_3$-axis. With slight abuse of notations, we will
keep the same notation to denote the restriction of these isometries to the
horizontal plane.

We define $z_0, \ldots, z_{m} \in S^1$ to be vertices of a regular polygon with $m
+1$ edges in the plane. Without loss of generality, we can choose
\[
z_0 : = (1,0) = e_1 \in \mathbb R^2 ,
\]
and, for $j=1, \ldots, m-1$, $z_{j+1} \in \mathbb R^2$ is the image of $z_j$ by $
\mathcal R_{m+1}$. In other words, if we identify the horizontal plane with $\mathbb
C$, the vertices of the polygon are exactly the $(m+1)$-th roots of unity. Recall that
the dihedral group of symmetries of $\mathbb R^2$ leaving this polygon fixed has
been denoted by ${\rm Dih}_{m+1}^{(2)}$. It is generated by $\mathcal R_{m+1}$
and $\mathcal S_2$.

Let $S^2$ be the unit sphere in $\mathbb R^3$. The upper half hemisphere of
$S^2$ can be parameterized by
\[
X^\upharpoonright (x) : = \left(x, \sqrt{1- |x|^2} \right) ,
\]
while the lower hemisphere is parameterized by
\[
X^\downharpoonright (x) : = \left( x, - \sqrt{1- |x|^2} \right) ,
\]
where, in both cases, $x \in D(0,1)$.

For all $\tau >0$ small enough, we set
\[
B^\upharpoonright : = X^\upharpoonright (D(0, \tau^{3/4})) ,
\]
and, for all $\rho >0$ satisfying
\begin{equation}
\frac{1}{C} \, \tau \leq \rho^2 \leq C \, \tau \, ,
\label{eq:rho}
\end{equation}
for some fixed constant $C >1$, we define
\[
B^\downharpoonright_j : = X^\downharpoonright (D( \rho \, z_j, \tau^{3/4})) ,
\]
for $j=0, \ldots, m$. We also define
\[
p^\upharpoonright : = X^\upharpoonright (0),
\]
to be the north pole of $S^2$ and, for $j=0, \ldots, m$, we define the points
\[
p^\downharpoonright_j : = X^\downharpoonright (\rho \, z_j) ,
\]
which are $m+1$ points arranged symmetrically near the south pole of $S^2$. By
construction $p^\downharpoonright_{j+1}$ is the image of $p^\downharpoonright_j
$ by $\mathcal R_{m+1}$.

\begin{definition}
We define $\mathfrak S_{\tau, \rho}$ to be the surface obtained by excising from
$S^2$, the sets $B^\upharpoonright $ and $B^\downharpoonright_j $, for $j=0,
\ldots, m$.
\end{definition}
Observe that, provided $\tau$ is chosen small enough, the surface $\mathfrak
S_{\tau, \rho}$ has $m+2$ boundaries. Moreover, this surface has been
constructed in such a way that it is invariant under the action of the dihedral group
${\rm Dih}_{m+1}^{(2)}$.

\subsection{The mean curvature of vertical graphs}

We recall some well known facts about the the mean curvature of vertical graphs in
$\mathbb R^3$. The mean curvature of the graph of the function $u$, namely the
surface parameterized by
\[
X(x) : = (x, u(x)) \in \mathbb R^3 ,
\]
where $x$ belongs to some open domain in $\mathbb R^2$, is given by
\[
M (u) := \frac{1}{2}æ\, \mbox{div} \left( \frac{\nabla u}{\sqrt{1+ |\nabla u|^2}} \right).
\]
Recall that the mean curvature is defined to be the average of the principal
curvatures and this explains the factor $1/2$.

It follows from this formula that the linearized mean curvature operator about the
graph of $u$ is given by
\[
\begin{array}{rllll}
DM (u) \, (v) & = & \displaystyle \frac{ \Delta v}{2 W} + \displaystyle \frac{3}{2 W^5}
\, ( \nabla u \cdot \nabla v) \, D^2 u \, [\nabla u, \nabla u] \\[3mm]
& - & \displaystyle \frac{1}{2 W^3} \, \left( ( \nabla u \cdot \nabla v) \, \Delta u + D^2
v \, [\nabla u , \nabla u ]
+ 2 \, D^2 u \, [\nabla u , \nabla v] \right) ,
\end{array}
\]
where
\[
W : = \sqrt{1+ |\nabla u|^2}æ,
\]
and where $D^2 f \, [æ\cdot, \cdot]$ is the second order differential of the function $f
$. One should be aware that $DM (u)$ is not the Jacobi operator $J_u$ about the
graph of the function $u$ since nearby surfaces are not parameterized as normal
graphs but are parameterized as vertical graphs over the horizontal plane. As
explained in \S 2, this operator and the Jacobi operator are conjugate and in fact,
assuming the the vertical graph is oriented to so that unit normal vector field points
upward, we have the relation
\[
X^* ( J_u \, w ) = DM(u) (W \, X^*w ) ,
\]
for any function defined on the graph of $u$.

Of interest will be the case where, for example,
\[
u (x) = \pm \sqrt{ 1 - |x|^2},
\]
where $x = (x_1, x_2) \in \mathbb R^2$. According to the sign chosen, the graph of
$u$ is the lower or the upper hemisphere of the sphere of radius $1$ centered the origin. In this
case, we have
\[
\nabla u (x) = \mp \frac{x}{\sqrt{1- |x|^2}}, \qquad \qquad \nabla^2 u (x) = \mp
\frac{ (1 - |x|^2) \, {\rm Id} + x \otimes x }{(1- |x|^2)^{3/2}} ,
\]
and
\[
\Delta u (x) = \mp \frac{2 - |x|^2}{(1- |x|^2)^{3/2}} .
\]
Using these, we find that the explicit expression of $DH(u)$ is given by
\begin{equation}
DM (u) \, w = \frac{1}{2} \, (1 - |x|^2)^{1/2} \, \left( \Delta w - \nabla^2 w \, (x,x) -
4 \, (x \cdot \nabla w) \right) ,
\label{eq:expl}
\end{equation}
in $D(0,1)$.

\subsection{Green's function}

Let $N_0$ denote the inward pointing unit normal vector field on $S^2$. We
consider an inward pointing vector field $N_0^\flat$ which is equal to $N_0$ close to the (horizontal)
equator of $S^2$ and which is equal to a vertical unit vector field close to the north
and south pole of $S^2$ (still pointing inward).

We define ${\mathbb L}$ to be the linearized mean curvature operator using the
vector field $N_0^\flat$. According to the analysis of \S 2, we can write
\begin{equation}
{\mathbb L} \, w : = \frac{1}{2} \, \left( \Delta_{S^2} +2 \right) ( N_0 \cdot N_0^\flat \,
w \, ) .
\label{eq:Letoile}
\end{equation}
We let $\Gamma_\rho$ be the unique solution of
\begin{equation}
{\mathbb L} \, \Gamma_\rho = - \pi \, \delta_{p^\upharpoonright} - \frac{\pi}{\sqrt {1-
\rho^2}} \, \frac{1}{m+1} \, (\delta_{p^\downharpoonright_0} + \ldots +
\delta_{p_{m}^\downharpoonright} ) ,
\label{eq:Gamma}
\end{equation}
which satisfies the orthogonality conditions
\[
\int_{S^2} x_i \, \Gamma_\rho \, {\rm dvol}_{S^2} = 0 ,
\]
for $i=1,2$ and $3$. Here $\delta_q$ is the Dirac mass at the point $q$. The
existence of $\Gamma_\rho$ is guarantied by the fact that the distribution on the
right hand side of (\ref{eq:Gamma}) is orthogonal to the cokernel of ${\mathbb L}$.
Indeed, the Jacobi operator is self-adjoint and its kernel and cokernel are equal
and spanned by the restriction of the coordinate functions to the unit sphere.
Thanks to (\ref{eq:Letoile}), we conclude that the cokernel of ${\mathbb L}$ is also
spanned by the restriction of the coordinate functions to the unit sphere. Now
\[
\langle x_1 , \delta_{p^\upharpoonright} \rangle_{\mathcal D, \mathcal D'} =
\langle x_2 , \delta_{p^\upharpoonright} \rangle_{\mathcal D, \mathcal D'} = 0 ,
\]
since both $x_1$ and $x_2$ vanish at the north pole of $S^2$ and
\[æ
\langle x_1 , \delta_{p^\downharpoonright_j} \rangle_{\mathcal D, \mathcal D'} =
\rho \, \cos \left( \frac{2\pi}{m+1} j \right)
\quad
\mbox{and}
\quad
\langle x_2 , \delta_{p^\downharpoonright_j} \rangle_{\mathcal D, \mathcal D'} =
\rho \, \sin \left( \frac{2\pi}{m+1} j \right) .
\]
Since
\[
\sum_{j=0}^m \cos \left( \frac{2\pi}{m+1} j \right) = \sum_{j=0}^m \sin \left(
\frac{2\pi}{m+1} j \right) = 0 ,
\]
we conclude that the distribution on the right hand side of (\ref{eq:Gamma}) is
orthogonal to the coordinate functions $x_1$ and $x_2$. Geometrically, this is
simply a consequence of the fact that the points $p_j^\downharpoonright$ are
symmetrically arranged around the $x_3$-axis. Finally, we have
\[
\langle x_3 , \delta_{p^\upharpoonright} \rangle_{\mathcal D, \mathcal D'} = 1,
\qquad
\mbox{and}
\qquadæ
\langle x_3 , \delta_{p^\downharpoonright_j} \rangle_{\mathcal D, \mathcal D'} =
\sqrt{1- \rho^2} ,
\]
and, again, we conclude that the distribution on the right hand side of
(\ref{eq:Gamma}) is orthogonal to the coordinate function $x_3$ thanks to the
choice of the constant in front of the Dirac masses at the points
$p_j^\downharpoonright$. Geometrically, this will have some interesting
consequence and can be interpreted as a conservation of the vertical flux of the
surfaces we try to construct. We shall return to this point later on. Finally, observe
that $\Gamma_\rho$ is invariant under the action of the elements of ${\rm Dih}_{m
+1}^{(2)}$.

The following result provides the expansion of the function $\Gamma_\rho$ close
to the north pole of $S^2$.
\begin{lemma}
The following expansion holds
\[
X^\upharpoonright \, ^* \Gamma_\rho ( x ) = - \log |x| + a^\upharpoonright +
{\mathcal O}_{\mathring{\mathcal C}^\infty} (|x|^2) ,
\]
in a fixed neighborhood of $0$, where the constant $a^\upharpoonright \in
\mathbb R$ depends smoothly on $\rho$ and is bounded as $\rho$ tends to $0$.
Moreover, the estimate on ${\mathcal O}_{\mathring{\mathcal C}^\infty} (|x|^2)$ is
uniform as $\rho$ tends to $0$.
\end{lemma}
\begin{proof}
We define the function $\Gamma_0$ on the upper hemisphere by
\[
X^\upharpoonright \, ^* \Gamma_0 (x) = -\log |x| ,
\]
and, using (\ref{eq:expl}), we compute
\[
X^\upharpoonright \, ^* ({\mathbb L} \, \Gamma_0 + \pi \, \delta_{p^
\upharpoonright} ) = \frac{3}{2} \, \sqrt{1 - |x|^2} .
\]
This immediately implies that, close to $p^\upharpoonright$, the function $
\Gamma_\rho - \Gamma_0$ is smooth. In particular, this function can be expanded
as
\[
X^\upharpoonright \, ^* (\Gamma_\rho - \Gamma_0) ( x ) = a^\upharpoonright +
b^\upharpoonright \cdot x + {\mathcal O}_{\mathring{\mathcal C}^\infty} (|x|^2) ,
\]
where $a^\upharpoonright \in \mathbb R$ and $b^\upharpoonright \in \mathbb
R^2$ depend smoothly on $\rho$ and remain bounded as $\rho$ tends to $0$.
Since the function $\Gamma_\rho$ is also invariant under the action of the
elements of ${\rm Dih}_{m+1}^{(2)}$, we conclude that necessarily $b^
\upharpoonright =0$. This completes the proof of the result.
\end{proof}

Near the other poles, the function $\Gamma_\rho$ also has an expansion which
we now describe. As can be suspected, this later description relies on the
expansion of the function
\[
G (x) : = - \sum_{j=0}^m \log | x - \rho \, z_j| ,
\]
at any of its singularities. Since this is a key point in our construction, we spend
some time to derive this expansion carefully. By symmetry, it is enough to expand
this function at $\rho \, z_0$. We change variables and write $x = \rho\, z_0 + y$.
We then expand
\[
\log | y - \rho \, (z_j - z_0)| = \log \rho + \log | z_j - z_0| + \frac{1}{\rho} \, \frac{z_0 -
z_j}{|z_0 - z_j|^2} \cdot y + {\mathcal O} \left( \frac{|y|^2}{\rho^2} \right) .
\]
Hence we find
\[
\begin{array}{rllll}
G ( \rho\, z_0 + y) & = & \displaystyle - \log |y| - m \, \log \rho - \sum_{j=1}^m \log |
z_j - z_0| \\[3mm]
& - & \displaystyle \frac{1}{\rho} \, \sum_{j=1}^m \frac{z_0 - z_j}{|z_0-z_j|^2} \cdot y +
{\mathcal O} \left( \frac{|y|^2}{\rho^2} \right) .
\end{array}
\]
It is easy to check that the following identity holds
\[
\sum_{j=1}^m \frac{z_0 - z_j}{|z_0-z_j|^2} = \frac{m}{2} \, z_0 .
\]
Setting
\[
a_0^\downharpoonright : = \sum_{j=1}^m \log | z_j - z_0| ,
\]
we can write
\[
G ( \rho\, z_0 + y) = - \log |x| - m \, \log \rho - a_0^\downharpoonright - \frac{m}{2
\rho} \, z_0 \cdot y + {\mathcal O} \left( \frac{|y|^2}{\rho^2} \right) .
\]
Similar estimates can be obtained for the partial derivatives of $G$. Finally,
observe that
\[
\Delta G = - 2 \, \pi \, \left( \delta_{\rho z_0} + \ldots + \delta_{\rho z_m}\right) .
\]

We now prove that, at $p_j^\downharpoonright $, the expansion of the function
$X^{\downharpoonright} \,^* \Gamma_\rho$ is (in some sense to be made precise)
close to the expansion of $G$ near $\rho \, z_j$. This is the content of the
following~:
\begin{lemma}
The following expansion holds
\[
\begin{array}{rllll}
X^\downharpoonright \, ^* \Gamma_\rho ( \rho \, z_j + y ) =\displaystyle -
\frac{1}{m+1} \, \left( \log |x|æ+ m \, \log \rho + a^\downharpoonright_{0,\rho} +
\frac{m}{2 \rho}æ \, z_j \cdot y \right) + \displaystyle {\mathcal
O}_{\mathring{\mathcal C}^\infty} (\tau^{1/2}) ,
\end{array}
\]
for $ |y| \in [\frac{1}{2} \, \tau^{3/4} , 2 \, \tau^{3/4}]$. Here
$a^\downharpoonright_{0,\rho} \in \mathbb R $ smoothly depends on $\rho >0$
and is uniformly bounded as $\rho$ tends to $0$.
\end{lemma}
\begin{proof} Thanks to the invariance with respect to the action of ${\rm Dih}_{m
+1}^{(2)}$, it is enough to describe this expansion near the point
$p^\downharpoonright_0$. As in the proof of the previous Lemma we show that,
near the south pole of $S^2$, the function
$X^\downharpoonright \, ^* \Gamma_\rho $ is not too far from $G$. To this aim,
we define $\tilde \Gamma_\rho$ on the lower hemisphere of $S^2$ by
\[
X^\downharpoonright \, ^* \tilde \Gamma_\rho = G ,
\]
and, thanks to (\ref{eq:expl}), we can compute
\[
\begin{array}{llll}
X^\downharpoonright \, ^* \left( {\mathbb L} \, \tilde \Gamma_\rho +
\displaystyle {\pi} \, \sqrt{1-\rho^2} \, (\delta_{p^\downharpoonright_0} + \ldots +
\delta_{p^\downharpoonright_{m}} ) \right) = \displaystyle \frac{1}{2} \, \sqrt{1-
|x|^2} \, \times \\[3mm]
\qquad \qquad \displaystyle \sum_{j=0}^m \left( 3 - 2 \, \rho \, \frac{z_j \cdot (x-
\rho \, z_j)}{|x- \rho \, z_j|^2} + \frac{\rho^2}{|x-\rho \, z_j|^2}æ \, \left( 1 - 2 \, \frac{ (z_j
\cdot (x- \rho \, z_j))^2}{|x- \rho \, z_j|^2} \right) \right) .
\end{array}
\]
Observe that the right hand side contains three terms which have different
regularity properties. The first one is a smooth function which depends smoothly on
$\rho$ and which is invariant by rotation. The second function has a singularity of
order $1$ at each $\rho \, z_j$ and is bounded by a constant times $\rho \, |x- \rho \,
z_j|^{-1}$. Finally, the third function has a singularity of order $2$ at each $\rho \,
z_j$ and is bounded by a constant times $\rho^2 \, |x- \rho \, z_j|^{-2}$.

As a consequence, $ X^\downharpoonright \, ^* ( \tilde \Gamma_\rho -
\Gamma_\rho)$ can be decomposed into the sum of three functions which can be
analyzed independently. The first one $f^{(1)}_\rho$ is smooth in a fixed
neighborhood of $0$ and depends smoothly on the parameter $\rho$. This implies
that, near each $\rho \, z_j$, this function has a Taylor's expansion with coefficients
smoothly depending on $\rho$. Hence
\[
f_\rho^{(1)} (x) = f_\rho^{(1)} (\rho \, z_0) + \nabla f_\rho^{(1)}(\rho \, z_0) \cdot (x -
\rho\, z_0) + \mathcal O (| x- \rho \, z_0|^2) .
\]
Observe that $ \nabla f_\rho^{(1)}(0) =0$ and hence $| \nabla f_\rho^{(1)}(\rho \,
z_0)| \leq C \, \rho$. We conclude that $f_\rho^{(1)} (x) = f_\rho^{(1)} (\rho \, z_0) +
\mathcal O (\tau^{5/4})$ when $| x- \rho \, z_0| \in [\frac{1}{2} \, \tau^{3/4} , 2 \,
\tau^{3/4}]$.

Since $\sum_j |z-z_j|^{-1} \in L^p (D(0,1/2))$ for all $p \in (1,2)$, we find that the
second function $f^{(2)}_\rho \in W^{2,p} (D(0,1/3))$ and hence that it is continuous
near $\rho \, z_0$ and $ f_\rho^{(2)} (x) - f_\rho^{(2)} (\rho \, z_0)$ is bounded by a
constant times $\rho \, \sum_{j=0}^m | x- \rho \, z_j |^{\nu}$, for any given $\nu < 1$.
In particular, $f_\rho^{(2)} (x) = f_\rho^{(2)} (\rho \, z_0)+ \mathcal O (\tau^{(2+ 3
\nu )/4})$ when $| x- \rho \, z_0| \in [\frac{1}{2} \, \tau^{3/4} , 2 \, \tau^{3/4}]$.

Finally, using the result of Proposition~\ref{pr:6.2}, the third function $f^{(3)}_\rho$
is bounded by a constant times $\rho^2 \, \sum_{j=0}^m |x- \rho \, z_j|^{\mu}$, for
any $\mu \in (-1,0)$.

In particular, when $| x- \rho \, z_0| \in [\frac{1}{2} \, \tau^{3/4} , 2 \, \tau^{3/4}]$, we
find that the sum of these function can be decomposed as the sum of a constant
function (smoothly depending on $\rho$) and a function which is bounded by a
constant times $\tau^{1/2}$ (chose $\nu =1/2$ and $\mu =-1/2$). The statement
then follows at once. \end{proof}

It is interesting to observe that $\Gamma_\rho$ depends on $\rho>0$ since the
points $p^\downharpoonright_j$ do and, as $\rho$ tends to $0$, the sequence $
\Gamma_\rho$ converges on compacts to the unique solution of
\[
{\mathbb L} \, \Gamma_0 = - \pi \, \left( \delta_{p^\upharpoonright} +
\delta_{p^\downharpoonright} \right) ,
\]
which is $L^2$-orthogonal to the smooth kernel of $\Delta_{S^2} +2$. Recall that
$p^\upharpoonright$ denotes the north pole of $S^2$ and we now agree that
$p^\downharpoonright$ denotes the south pole of $S^2$.

\begin{remark}
For later use, it will be important to notice that all solutions of ${\mathbb L} \, w =0$
which are defined in $S^2- \{ p^\upharpoonright , p^\downharpoonright \}$, are
invariant under the action of ${\rm Dih}_{m+1}^{(2)}$ and are bounded by a
constant times ${\rm dist} (\cdot , \{ p^\upharpoonright ,
p^\downharpoonright \})^\nu$ for some $\nu \in (-1, 0)$, are linear combinations of
the functions $x_3$ and $\Gamma_0$.
\label{rem-m}
\end{remark}

We now summarize the above analysis. We set
\[
u^\upharpoonright (x) : = \sqrt{1 - |x|^2} ,
\quad \mbox{and} \quad
u^\downharpoonright (x) : = - \sqrt{1 - |x|^2} .
\]
Observe that, thanks to the previous results, we see that near $0$, the graph of the
function
\[
v^\upharpoonright : = u^\upharpoonright + \tau \, X^\upharpoonright \,
^* \Gamma_\rho ,
\]
can be expanded
\[
v^\upharpoonright ( x ) = \displaystyle 1 + \tau \, (m \, \log \rho +
a^\upharpoonright ) + \displaystyle \tau \, \log |x| +
{\mathcal O}_{\mathring {\mathcal C}^\infty} (\tau^{3/2}) ,
\]
for $|x| \in [\frac 1 2 \, \tau^{3/4} , 2 \, \tau^{3/4}]$, where $ a^\upharpoonright \in
\mathbb R$ smoothly depends on $\rho$. Moreover, we see that near $\rho \, z_j$,
the graph of the function
\[
v^\downharpoonright : = u^\downharpoonright + \tau \, X^\downharpoonright \,
^* \Gamma_\rho ,
\]
can be expanded
\[
\begin{array}{rllll}
v^\downharpoonright ( \rho \, z_j + y ) & = & \displaystyle - \sqrt{1-\rho^2} -
\frac{\tau}{m+1} \, (m \, \log \rho + a^\downharpoonright ) - \displaystyle
\frac{\tau}{m+1} \, \log |y| \\[3mm]
& - & \displaystyle æ\left( \rho - \frac{m}{m+1} \, \frac{\tau}{ 2 \rho}æ \right) \, z_j \cdot
y + {\mathcal O}_{\mathring {\mathcal C}^\infty} (\tau^{3/2}) ,
\end{array}
\]
for $|y| \in [\frac 1 2 \, \tau^{3/4} , 2 \, \tau^{3/4}]$, where $ a^\downharpoonright \in
\mathbb R$ smoothly depends on $\rho$. The key point in our construction is that
the constant in front of $z_j \cdot y$ can be adjusted by choosing $\rho$
appropriately. Indeed, if we define $\rho_0 >0$ by the identity
\[
2 \, (m+1) \, \rho_0^2 = m \, \tau ,
\]
then, when $\rho = \rho_0$, the constant in front of $z_j \cdot y$ in the last
expansion is exactly $0$ while choosing $\rho \neq \rho_0$ slightly larger or
smaller allows one to prescribe any value of this constant, close enough to $0$.

\subsection{Mapping properties of the Jacobi operator about a punctured sphere}

To begin with we define on $S^2$, the distance function to the punctures
$p^\upharpoonright , p_0^\downharpoonright , \ldots, p_m^\downharpoonright$
by
\[
d : = \mbox{dist}_{S^2} \left( \, \cdot , \{p^\upharpoonright ,
p_0^\downharpoonright , \ldots , p_m^\downharpoonright \} \right) .
\]
Even though this is not apparent in the notations, the function $d$ depends
implicitly on $\rho$ since it depends on the location of the points
$p_j^\downharpoonright$ which themselves do depend on $\rho$. We can define
some weighted spaces on
\[
S^* := S^2 - \{æp^\upharpoonright , p_0^\downharpoonright , \ldots,
p_m^\downharpoonright\} .
\]

For all $\nu \in \mathbb R$ and $k \in \mathbb N$ we define $\mathcal C^{k,
\alpha}_\nu (S^*)$ to be the space of functions $w \in \mathcal C^{k, \alpha}_{loc}
(S^*)$ for which the following norm is finite
\[
\begin{array}{rllll}
\| w \|_{\mathcal C^{k, \alpha}_\nu (S^*)} & : = & \displaystyle \sum_{j=0}^k \sup_{p
\in S^*} d^{-\nu +j} (p) \, \| \nabla^j w (p)\|_{g_{S^2}} \\[3mm]
& + & \displaystyle \sup_{\zeta \in (0, \pi/2)} \sup_{ d(p), d(q) \in [\zeta, 2\zeta]}
\zeta^{-\nu + k+ \alpha} \, \frac{\| \nabla^k w (p) -
\nabla^k w(q)\|_{g_{S^2}}}{ \mbox{dist}_{S^2} (p,q)^{\alpha}} .
\end{array}
\]
We further assume that the functions in $\mathcal C^{k, \alpha}_\nu (S^*)$ are
invariant under the action of ${\rm Dih}_{m+1}^{(2)}$. Again, observe that the
weighted spaces $\mathcal C^{k, \alpha}_\nu (S^*)$ do implicitly depend on $\rho
$.

We consider the operator
\[
\begin{array}{rcccllll}
{\mathbb L}_\nu : & \mathcal C^{2, \alpha}_\nu (S^*) & \longrightarrow
& \mathcal C^{0, \alpha}_{\nu -2} (S^*) \\[3mm]
& w & \longmapsto & {\mathbb L} \, w .
\end{array}
\]
It is easy to check that ${\mathbb L}_\nu$ is well defined.

Recall that ${\mathbb L}$ is conjugate to $\Delta_{S^2} +2$. When acting on
smooth function defined on $S^2$, the mapping properties of $\Delta_{S^2} +2$
are well understood and we recall that the kernel of this operator is spanned by the
restriction to $S^2$ of the linear forms on $\mathbb R^3$. Since we are assuming
that the functions we consider are invariant under the action of the dihedral group $
{\rm Dih}_{m+1}^{(2)}$, this implies that the bounded kernel of $\mathbb L$ has
dimension $1$. We now investigate the mapping properties of ${\mathbb L}$ (or
alternatively $\Delta_{S^2} +2$) when acting on functions belonging to the
weighted spaces we have just defined. We start with the~:
\begin{proposition}
Assume that $\nu \in (-1,0)$, then there exist constants $ C, \rho_0 >0$ only
depending on $\nu$ such that, for all $\rho \in (0, \rho_0)$, we have
\[
\| w\|_{\mathcal C^{2, \alpha}_\nu(S^*)} \leq C \, \| {\mathbb L}\, w\|_{\mathcal C^{0,
\alpha}_\nu(S^*)} ,
\]
for all functions $w$ in the $L^2(S^2)$-orthogonal complement of the functions
$x_3$ and $\Gamma_\rho$.
\label{pr:6.22}
\end{proposition}
\begin{proof} As usual, thanks to Schauder's estimates, it is enough to prove that
\[
\|æd^{-\nu} \, w\|_{L^\infty (S^*)} \leq C \, \|æd^{2-\nu} \, {\mathbb L} \, w\|_{L^\infty
(S^*)} ,
\]
for all $\rho$ small enough.

As usual, the proof of this estimate is by contradiction. Assume that the estimate is
not true, then, there would exist a sequence $\rho_n$ tending to $0$ and a
sequence of functions $w_n$ such that
\[
\|æd^{-\nu} \, w_n \|_{L^\infty (S^*)} = 1
\qquad \mbox{and} \qquad \lim_{n \rightarrow \infty} \, \|æd^{2-\nu} \, {\mathbb L} \,
w_n\|_{L^\infty (S^*)} =0 .
\]
Moreover $w_n$ is invariant under the action of ${\rm Dih}_{m+1}^{(2)}$ and is
$L^2$-orthogonal to $x_3$ and $\Gamma_{\rho_n}$ (recall that $\Gamma_\rho
= \Gamma_{\rho_n}$ depends on $\rho_n$). Hence,
\begin{equation}
\int_{S^2} \, x_3 \, w_n \, {\rm dvol}_{S^2} =0 ,
\label{eq:or1}
\end{equation}
and
\begin{equation}
\int_{S^2} \, \Gamma_{\rho_n} \, w_n \, {\rm dvol}_{S^2} =0 .
\label{eq:or2}
\end{equation}

We choose a point $q_n \in S^*$ such that
\[
|w_n (q_n)| \geq 1/2 \, d^{\nu} (q_n) ,
\]
and we distinguish various cases according to the behavior of the sequence
$d(q_n)$. In each case, we rescale coordinates (using the exponential map) by $1/
d(q_n)$ and we use elliptic estimates together with Ascoli-Arzela's theorem to extract
from the sequence $\tilde w_n : = d^{-\nu} (q_n) \, w_n$ convergent subsequences.
Finally, we pass to the limit in the equation satisfied by $\tilde w_n$. If, for some
subsequence, $d(q_n)$ remains bounded away from $0$, we get in the limit a non
trivial solution of
\[
(\Delta_{S^2} + 2) \, w = 0 ,
\]
which is defined in $S^2 - \{p^\upharpoonright, p^\downharpoonright \}$, where we
recall that $p^\upharpoonright$ denotes the north pole and $p^\downharpoonright
$ denotes the south pole of $S^2$. Moreover, $w$ is bounded by a constant times
$({\rm dist}( p, \{p^\upharpoonright, p^\downharpoonright \}))^{\nu}$ and $w$ is
invariant under the action of ${\rm Dih}_{m+1}^{(2)}$. Finally, we can pass to the
limit in (\ref{eq:or1}) and (\ref{eq:or2}) and check that $w$ is $L^2$-orthogonal to
$x_3$ and $\Gamma_0 : = \lim_{n\rightarrow \infty} \Gamma_{\rho_n}$. It is easy to
check (see Remark~\ref{rem-m}) that this implies that $w \equiv 0$, which is a
contradiction.

The second case we have to consider is the case where $\lim_{n\rightarrow \infty}
d(q_n) = 0$ and $\lim_{n\rightarrow \infty} d(q_n) / \rho_n = + \infty$ or the case
where $\lim_{n\rightarrow \infty} d(q_n) / \rho_n = 0$. In either case, we obtain a
nontrivial solution of
\[
\Delta \, w = 0 ,
\]
in $\mathbb R^2 - \{0\}$ which is bounded by a constant times ${\rm dist} (\cdot,
\{0\} )^{\nu}$. It is easy to check that $w\equiv 0$ since $\delta \notin \mathbb Z$,
which is again a contradiction.

Finally, we consider the case where $\lim_{n\rightarrow \infty} d(q_n) / \rho_n$
exists. In this case, we obtain a nontrivial solution of
\[
\Delta \, w = 0 ,
\]
in $\mathbb R^2 - \{r_0 \, z_0, \ldots, r_0 \, z_{m}\}$, for some $r_0 >0$. Moreover,
we know that this solution is bounded by a constant times $({\rm dist} (\cdot, \{r_0
\, z_0, \ldots, r_0 \, z_{m}\}))^{\nu}$ and $w$ is also invariant under the action of $
{\rm Dih}_{m+1}^{(2)}$. Inspection of the behavior of $w$ at the points $r_0 \, z_j$
together with the fact that $\nu > -1$ and $w$ is invariant with respect to the action
of ${\rm Dih}_{m+1}^{(2)}$, implies that $w$ is a solution in the sense of
distributions of
\[
\Delta \, w = a \, \sum_{j=0}^m \delta_{r_0 \, z_j} ,
\]
for some $a \in \mathbb R$. Then, inspection of $w$ at infinity together with the fact
that $\nu <0$, implies that necessarily $a =0$ and hence $w \equiv 0$. This is
again a contradiction. Having reached a contradiction in each case, this completes
the proof of the result. \end{proof}

Thanks to the previous result, we can prove the~:
\begin{proposition}
Assume that $\nu \in (-1,0)$ is fixed. Then the operator ${\mathbb L}_\nu$ is
surjective and has a $2$ dimensional kernel spanned by the functions $x_3$ and $
\Gamma_\rho$. . Moreover, the right inverse of ${\mathbb L}_\nu$ which is
chosen so that its image is in the $L^2$-orthogonal complement of the kernel of $
{\mathbb L}_\nu$, has norm which is bounded independently of $\rho$ small
enough.
\label{pr:6.2}
\end{proposition}
\begin{proof} The existence of a right inverse follows from the general theory
developed for example in \cite{Pac}. Nevertheless, we give here a self-contained
proof.

Let us assume that we are given a function $f \in \mathcal C^{0, \alpha} (S^*)$
which has compact support in $S^*$. Recall that the functions we are interested in
are invariant under the action of ${\rm Dih}^{(2)}_{m+1}$. We choose $a \in
\mathbb R$ so that $ f - a \, \delta_{p^\upharpoonright}$ is orthogonal to the
function $x_3$. In particular, this implies that we can solve
\[
\mathbb L \, \tilde w = f - a \, \delta_{p^\upharpoonright}
\]
and, choosing the constant $b \in \mathbb R$ appropriately, we can assume that
$w : = \tilde w - b \, \Gamma_\rho$ is $L^2$-orthogonal to the function $x_3$ and $
\Gamma_\rho$. Observe that
\[
\mathbb L \, w = f ,
\]
in $S^*$ and also that $w \in \mathcal C^{2, \alpha}_\nu (S^*)$. In particular the
result of Proposition~\ref{pr:6.22} applies and we have
\[
\| w\|_{\mathcal C^{2, \alpha}_\nu(S^*)} \leq C \, \| {\mathbb L}\, w\|_{\mathcal C^{0,
\alpha}_\nu(S^*)} .
\]
The general result, when $f$ is not assumed to have compact support in $S^*$ can
be handled as usual using a sequence of functions having compact support and
converging on compacts to a given function in $\mathcal C^{0, \alpha}_\nu(S^*)$.
\end{proof}

\subsection{A third fixed point argument}

Assume that we are given $\tau , \tilde \tau >0$ small enough and satisfying
\begin{equation}
| \tilde \tau - \tau | \leq \kappa \, \tau^{3/2} ,
\label{eq:ttt}
\end{equation}
where the constant $\kappa >0$ is fixed large enough and will be fixed in the last
section of the paper. We also assume that $\rho >0$ satisfies
\begin{equation}
\left| \rho - \frac{m}{m+1} \, \frac{\tau}{ 2 \rho}æ \right| \leq \kappa \, \tau^{3/4} .
\label{eq:6.30}
\end{equation}

We prove the existence of an infinite dimensional family of constant mean
curvature surfaces which are close to $\mathfrak S_{\tau, \rho}$ are parameterized
by their boundary values described by two functions $f^\upharpoonright : S^1
\longrightarrow \mathbb R $ and $f^\downharpoonright : S^1 \longrightarrow
\mathbb R$. The surfaces also depend on $\tilde \tau$ and $\rho$ satisfying the
above estimates.
\begin{proposition}
Assume we are given $\kappa >0$ large enough (the value of $\kappa$ will be
fixed later on). For all $\tau, \tilde \tau >0$ small enough satisfying (\ref{eq:ttt}) and
for all functions $f^\upharpoonright$ which are invariant under the action of the
dihedral group ${\rm Dih}^{(2)}_{m+1}$ and $f^\downharpoonright$, which are
invariant under the action of $\mathcal S_2$, both satisfying (H1)
and
\[
\|æf \|_{{\mathcal C}^{2,\alpha} (S^1)} \leq \kappa \, \tau^{3/2} ,
\]
there exists a constant mean curvature surface $\mathfrak S_{\tilde \tau , \rho ,
f^\upharpoonright, f^\downharpoonright}$ which is a graph over $\mathfrak S_{\tau,
\rho}$, has $m+2$ boundaries (one boundary close to the north pole and $m+1$
boundaries close to the south pole) and is invariant under the action the dihedral
group ${\rm Dih}_{m+1}^{(2)}$. Close to the upper boundary, the surface $
\mathfrak S_{\tilde \tau ,\rho , f^\upharpoonright, f^\downharpoonright}$ a {\em
vertical graph} over the annulus
\[
\{ x \in \mathbb R^2 \, : \, \tau^{3/4} \leq |x | \leq 2 \, \tau^{3/4} \, \} ,
\]
for some function $x \longmapsto V^\upharpoonright_{\tilde \tau , \rho ,
f^\upharpoonright, f^\downharpoonright} ( \tau^{-3/4} \, x )$ which can be expanded
as follows
\begin{equation}
V^\upharpoonright_{\tilde \tau , \rho, f^\upharpoonright, f^\downharpoonright} (x) =
\displaystyle 1 + \tilde \tau \, (m \, \log \rho + a^\upharpoonright_{\tilde \tau , \rho ,
f^\upharpoonright, f^\downharpoonright} ) + \frac{3}{4} \,
\tilde \tau \, \log \tau + \tilde \tau \, \log |x| + W^{\rm out}_{f^\upharpoonright} (x) +
\bar V^\upharpoonright_{\tilde \tau , \rho , f^\upharpoonright, f^\downharpoonright}
( x ) ,
\label{eq:6.31}
\end{equation}
where $a^\upharpoonright \in \mathbb R$, $W^{\rm out}_{f}$ denotes the bounded
harmonic extension of $f$ in $\mathbb R^2 - \overline D(0,1)$ and where
\begin{equation}
\| \bar V^\upharpoonright_{\tilde \tau , \rho , 0, 0} \|_{\mathcal C^{2, \alpha}
( \overline D(0, 2) - D (0, 1) )} \leq C \, \tau^{3/2} ,
\label{eq:6.321}
\end{equation}
and, , given $\nu \in (-1,0)$,
\begin{equation}
\begin{array}{llll}
\| \bar V^\upharpoonright_{\tilde \tau , \rho , f^\upharpoonright,
f^\downharpoonright } - \hat V^\upharpoonright_{\tilde \tau , \rho , f^\upharpoonright
\, ', f^\downharpoonright \, '} \|_{\mathcal C^{2, \alpha} ( \overline D(0, 1) - D (0,
1/2) )} \qquad \qquad \\[3mm]
\qquad \qquad \leq \, C \, \tau^{(1+\nu)/4} \, (\| f^\upharpoonright \, ' -
f^\upharpoonright \|_{\mathcal C^{2, \alpha} (S^1)} + \| f^\downharpoonright \, ' -
f^\downharpoonright \|_{\mathcal C^{2, \alpha} (S^1)} ),
\end{array}
\label{eq:6.32}
\end{equation}
for some constant $C >0$ independent of $\kappa$, $
\tilde \tau$ and $f^\upharpoonright, f^\downharpoonright, f^\upharpoonright \, ' ,
f^\downharpoonright \, '$.

Near one of the lower boundaries the surface $\mathcal S_{\tilde \tau , \rho ,
f^\upharpoonright, f^\downharpoonright}$ is a {\em vertical graph} over the annulus
\[
\{ x \in \mathbb R^2 \, : \, \tau^{3/4} \leq |x - \rho \, z_0| \leq 2 \, \tau^{3/4} \, \}
\]
for some function $x \longmapsto V^\downharpoonright_{\tilde \tau , \rho,
f^\upharpoonright, f^\downharpoonright}(\tau^{3/4} \, (x - \rho \, z_0))$ which can be
expanded as follows
\begin{equation}
\begin{array}{rlllll}
V^\downharpoonright_{\tilde \tau , \rho , f^\upharpoonright, f^\downharpoonright}
(x) & = & \displaystyle - \sqrt{1-\rho^2} - \frac{\tilde \tau}{m+1} \, (m \, \log \rho + a^
\downharpoonright_{\tilde \tau , \rho , f^\upharpoonright, f^\downharpoonright}
) - \displaystyle \frac{3\tilde \tau}{4(m+1)} \, \log \tau \\[3mm]
& - & \displaystyle \frac{\tilde \tau}{m+1} \, \log |x| - \displaystyle \tau^{3/4}\,
æ\left( \rho - \frac{m}{m+1} \, \frac{\tilde \tau}{ 2 \rho}æ \right) \, z_0 \cdot x \\[3mm]
& + & W^{\rm out}_{f^\downharpoonright} (x) + \bar V^\downharpoonright_{\tilde
\tau , f^\upharpoonright, f^\downharpoonright} (x ) ,
\end{array}
\label{eq:6.33}
\end{equation}
where $\bar V^\downharpoonright_{\tilde \tau , \rho , f^\upharpoonright,
f^\downharpoonright} $ enjoys properties similar to the one described above for $
\bar V^\downharpoonright_{\tilde \tau , \rho , f^\upharpoonright,
f^\downharpoonright} $. Moreover, both depend continuously on $\tilde \tau$ and $
\rho$.
\label{pr:6.3}
\end{proposition}
\begin{proof}
Again the arguments of the proof are close to the one already performed in the
previous sections. The equation we try to solve can be written formally as
\begin{equation}
{\mathbb L} ( \tilde \tau \, \Gamma_\rho + \hat F + w ) = Q (\tilde \tau \, \Gamma_\rho
+ \hat F + w)
\label{rede}
\end{equation}
where $Q$ collects all the nonlinear terms. Here $\hat F$ is a function which can
be described as follows~: near the north pole $p^\upharpoonright$
\[
X^\upharpoonright \, ^* \hat F (x) = \chi \, W^{\rm out}_{f^\upharpoonright}
(\tau^{-3/4} \, x) ,
\]
where $\chi$ is a cutoff function identically equal to $1$ in $D(0, 1/4)$ and
identically equal to $1$ outside $D(0, 1/2)$. Near the south pole
$p^\downharpoonright$
\[
X^\downharpoonright \, ^* \hat F (x) = \sum_{j= 0}^m \bar \chi \left( \frac{x - \rho \,
z_j}{\rho} \right) \, W^{\rm out}_{f^\upharpoonright} (x - \rho \, z_j) ,
\]
where $\bar \chi$ is a cutoff function identically equal to $1$ in $D(0, c)$ and
identically equal to $1$ outside $D(0, c/2)$. Here $c = \sin (\pi /(m+1))$ so that the
balls of radius $c$ centered at the points $z_j$, for $j=0, \ldots, m$ are disjoint.

We choose
\[
\hat {\mathcal E}_\tau : \mathcal C^{0, \alpha} (\mathfrak S_{\tau,
\rho})æ\longrightarrow \mathcal C^{0, \alpha} (S^*) ,
\]
an extension operator such that
\[
\left\{
\begin{array}{rllll}
\hat {\mathcal E}_\tau (\psi) & = & \psi \qquad & \mbox{in}æ\qquad \mathfrak S_{\tau,
\rho}æ \\[3mm]
\hat {\mathcal E}_\tau (\psi) & = & 0 \qquad & \mbox{in} \qquad X^\upharpoonright
(D(0, \tau^{3/4}/2)) \cup \, \displaystyle {\bigcup}_{j=0}^m X^\downharpoonright
(D(\rho\, z_j, \tau^{3/4}/2)),
\end{array}
\right.
\]
and
\[
\| \hat{\mathcal E}_\tau (\psi) \|_{\mathcal C^{0, \alpha}_\nu (S^* ) } \leq C \, \| \psi
\|_{\mathcal C^{0, \alpha}_\nu (\mathfrak S_{\tau, \rho}) } .
\]
By definition, the norm in the space $\mathcal C^{0, \alpha}_\nu (\mathfrak S_{\tau,
\rho})$ is defined exactly as the norm in $\mathcal C^{0, \alpha}_\nu (S^*)$ but
points are restricted to $\mathfrak S_{\tau, \rho}$ instead of $S^*$.

We rewrite (\ref{rede}) as
\begin{equation}
\mathbb L\, w = \hat {\mathcal E}_\tau \, \left( - \mathbb L \, \hat F + Q \left( \tilde
\tau \, \Gamma_\rho + \hat F + w \right) \right) .
\label{eq:mce2}
\end{equation}
Observe that, by construction $\mathbb L \, ( \tilde \tau \, \Gamma_\rho) = 0$ away
from the singular points.

Again, on the right hand side it is understood that we consider the image by $\hat
{\mathcal E}_\tau $ of the restriction of the functions to $\mathfrak S_{\tau, \rho}$.

We assume that $\nu \in (-1,0)$ is fixed. It is easy to check that there exists a
constant $c >0$ (independent of $\kappa$) and a constant $c_\kappa >0$
(depending on $\kappa$) such that
\[
\left\| \hat {\mathcal E}_\tau \left( Q \left( \tilde \tau \, \Gamma_\rho \right) \right)
\right\|_{\mathcal C^{2, \alpha}_{\nu -2} (S^*) } \leq c \, \tau^{(6 -3\nu)/4} ,
\]
\[
\left\| \hat {\mathcal E}_\tau \left( \mathbb L \, \hat F \right) \right\|_{\mathcal C^{2,
\alpha}_{\nu -2} (S^*) } \leq c \, \tau^{(1-2\nu)/4} \, \left( \|æf^\upharpoonright
\|_{\mathcal C^{2, \alpha} (S^1)} + \|æf^\downharpoonright \|_{\mathcal C^{2, \alpha}
(S^1)} \right) ,
\]
and
\[
\begin{array}{llll}
\displaystyle \left\| \hat {\mathcal E}_\tau \left( Q \left( \tilde \tau \, \Gamma_\rho +
\hat F' + w' \right) - Q \left( \tilde \tau \, \Gamma_\rho + \hat F + w \right) \right)
\right\|_{\mathcal C^{2, \alpha}_{\nu -2} (S^*)} \\[3mm]
\quad \leq c_\kappa \, \left( \tau \, \| w' -w \|_{\mathcal C^{2, \alpha}_\nu (S^*)} +
\tau^{(4-3\nu)/4} \, \left( \|æf^\upharpoonright \, ' - f^\upharpoonright \|_{\mathcal
C^{2, \alpha} (S^1)} + \|æf^\downharpoonright \, '- f^\downharpoonright \|_{\mathcal
C^{2, \alpha} (S^1)} \right) \right)
\end{array}
\]
provided $w$ and $w'$ satisfy
\[
\| w \|_{\mathcal C^{2, \alpha}_\nu (S^*)}æ + \| w' \|_{\mathcal C^{2, \alpha}_\nu
(S^*)}æ\leq C \, \tau^{(6-3\nu)/4} ,
\]
for some fixed constant $C >0$ indendent of $\kappa$. Here $\hat F$ and $\hat F'$
are respectively associated to the harmonic extensions of the boundary data $
f^\upharpoonright, f^\downharpoonright$ and $f^\upharpoonright \, ',
f^\downharpoonright \, '$.

Now, we make use of the result of Proposition~\ref{pr:6.2} to rephrase the problem
as a fixed point problem and the previous estimates are precisely enough to solve
this nonlinear problem using a fixed point agument for contraction mappings in the
ball of radius $C_\kappa \, \tau^{(6-3\nu)/4}$ in $\mathcal C^{2, \alpha}_\nu (S^*)$,
where $C_\kappa$ is fixed large enough. Then, for all $\tau >0$ small enough, we
find that there exists a constant $C_\kappa >0$ (depending on $\kappa$) such
that, for all functions $f^\upharpoonright, f^\downharpoonright$ satisfying the above
hypothesis, there exists a solution $w$ of (\ref{rede}) satisfying
\[
\| w\|_{\mathcal C^{2, \alpha}_\nu (S^*)}æ\leq C \, \tau^{(6-3 \nu)/4} .
\]
In addition, we have the estimate
\[
\| w' - w \|_{ \mathcal C^{2, \alpha}_\nu (S^*)}æ\leq C_\kappa \, \tau^{(1-2\nu)/4} \,
\left( \|æf^\upharpoonright \, ' - f^\upharpoonright \|_{\mathcal C^{2, \alpha} (S^1)} +
\|æf^\downharpoonright \, '- f^\downharpoonright \|_{\mathcal C^{2, \alpha} (S^1)}
\right) ,
\]
for some constant $C>0$, which does not depend on $\kappa$ or $\tau$, where $w
$ (resp. $w'$) is the solution associated to $f^\upharpoonright, f^\downharpoonright
$ (resp. $f^\upharpoonright \, ', f^\downharpoonright \, '$).

The solution of (\ref{rede}) is obtained through a fixed point theorem for contraction
mappings, and it is classical to check that the solution we obtain depends
continuously on the parameters of the construction. In particular, the constant mean
curvature surface we obtain depends continuously on $\tilde \tau$ and $\rho$.
\end{proof}

\section{Connecting the pieces together}

We keep the notations of the previous sections. We assume that $\kappa >0$ is
fixed large enough (the value will be decided shortly) and assume that $\tau >0$
is chosen small enough so that all the results proven so far apply.

For all $\tilde x \in \mathbb R^2$, we define the annuli
\[
A^{out}_\tau (\tilde x) : = \{ x \in \mathbb R^2 \, : \, \tau^{3/4} \leq |x - \tilde x | \leq 2 \, \tau^{3/4} \, \} ,
\]
and
\[
A^{ins}_\tau (\tilde x) : = \{ x \in \mathbb R^2 \, : \, \tfrac{1}{2} \, \tau^{3/4} \leq |x - \tilde x | \leq \tau^{3/4} \, \}.
\]
Recall that a function $f$ defined on $S^1$ is said to satisfy (H1) if
\[
\int_{S^1} f (\theta) \, d\theta = 0.
\]
and it is said to satisfy (H2) if
\[
\int_{S^1} f (\theta) \, \cos \theta \, d\theta = \int_{S^1} f (\theta) \, \sin \theta \, d\theta =
0.
\]
Also recall that a function $f$ defined on $S^1$ is invariant under the action of ${\rm
Dih}^{(2)}_{m+1}$ if
\[
f\left( \theta + \frac{2\pi}{m+1}\right) = f(\theta)
\]
for all $\theta \in S^1$ and $f$ is invariant under the action of the symmetry $
\mathcal S_2$ if
\[
f(-\theta) = f(\theta)
\]
for all $\theta \in S^1$.

We now describe the different pieces of constant mean
curvature surfaces we have at hand.
\begin{itemize}

\item[(i)] Assume that we are given $f^\upharpoonright \in \mathcal
C^{2, \alpha} (S^1)$ which is invariant under the action of
${\rm Dih}^{(2)}_{m+1}$, satisfies (H1) and
\[
\|æf^\upharpoonright \|_{\mathcal C^{2, \alpha} (S^1)} \leq \kappa \, \tau^{3/2} .
\]
The result of Proposition~\ref{pr:4.4} provides a constant mean curvature (equal to
$1$) surface $\mathfrak D_{\tau ,f}^+$ which is invariant under the action of the
dihedral group ${\rm Dih}_{m+1}^{(2)}$, has one end asymptotic to the end of $
\mathfrak D_\tau^+$ and which, close to its boundary, can be parameterized as the
vertical graph of $x \longmapsto U^\upharpoonright ( \tau^{-3/4} \, x)$ over
$A^{ins}_\tau (0)$, where
\[
U^\upharpoonright (x) = c^\upharpoonright + \tau \, \log |x| - W^{\rm ins}_f (x) +
\bar U^\upharpoonright (x) , \label{eq:step1-1}
\]
where
\[
c^\upharpoonright : = \tau \, \log \left( \frac{2}{\tau^{1/4}} \right) \in \mathbb R,
\]
and where $\bar U^\upharpoonright$ satisfies (\ref{eq:est-1})
and (\ref{eq:est-2}). To simplify the notations we have not mentioned
the data $\tau, f$ in the notation for $U^\upharpoonright$ and
$\bar U^\upharpoonright$\\

\item[(ii)] Next, we assume that we are given $\tau_1>0$ satisfying
\[
|æ\tau_1 - \tau| \leq \kappa \, \tau^{3/2},
\]
and $\rho_1 >0$ satisfying
\[
\left| \rho_1 - \frac{m}{m+1} \, \frac{\tau}{2\rho_1} \right| \leq \kappa \, \tau^{3/4} .
\]
Further assume that we are given a function $f_1^\upharpoonright \in \mathcal
C^{2, \alpha} (S^1)$ invariant under the action of the dihedral group ${\rm Dih}_{m
+1}^{(2)}$ and a function $f^\downharpoonright_1 \in \mathcal C^{2, \alpha} (S^1)$,
invariant under the action of the symmetry $\mathcal S_2$, both satisfying (H1) and
\[
\|æf_1^\upharpoonright \|_{\mathcal C^{2, \alpha} (S^1)} \leq \kappa \, \tau^{3/2}
\qquad \mbox{and} \qquad \|æf_1^\downharpoonright \|_{\mathcal C^{2, \alpha}
(S^1)} \leq \kappa \, \tau^{3/2} .
\]
The result of Proposition~\ref{pr:6.3} provides a constant mean curvature (equal to
$1$) surface $\mathfrak S_{\tau_1 , \rho_1 , f^\upharpoonright_1,
f^\downharpoonright_1}$ which is invariant under the action of the dihedral group $
{\rm Dih}_{m+1}^{(2)}$ and which, close to its upper boundary can be parameterized
as the vertical graph of $x \longmapsto V^\upharpoonright ( \tau^{-3/4} \, x )$ over
$A^{out}_\tau (0)$ where
\[
\begin{array}{rllll}
V^\upharpoonright (x) = 1+ d^\upharpoonright+ \tau_1 \, \log |x| + W^{\rm out}_{f^
\upharpoonright_1} (x) + \bar V^\upharpoonright ( x ) ,
\end{array}
\label{eq:step2-1}
\]
\[
d^\upharpoonright : = \tau_1\, \, \left( m \, \log \rho_1 + a^\upharpoonright_{{\tau_1
, \rho , f^\upharpoonright_1, f^\downharpoonright_1}} + \frac{3}{4} \,
\log \tau \right) \in \mathbb R,
\]
and where $\bar V^\upharpoonright $ satisfies (\ref{eq:6.321}) and (\ref{eq:6.32}).
Close to one of its lower boundaries, this surface can be parameterized as a vertical
graph for some function $x \longmapsto V^\downharpoonright ( \tau^{-3/4} \,
(x - \rho_1 \, z_0) )$ over $A^{out}_\tau (\rho_1 \, z_0)$ which can be expanded as
\[
\begin{array}{rlllll}
V^\downharpoonright
(x) & = & \displaystyle - 1 + c^\downharpoonright - \displaystyle \frac{\tau_1}{m+1}
\, \log |x| - \displaystyle \tau^{3/4} \, æ\left( \rho_1 - \frac{m}{m+1} \, \frac{\tau_1}{ 2
\rho}æ \right) \, z_0 \cdot x \\[3mm]
& + & W^{\rm out}_{f^\downharpoonright_1} (x) + \bar V^\downharpoonright (x ) ,
\end{array}\label{eq:step2-2}
\]
\[
c^\downharpoonright : = 1 - \sqrt{1-\rho_1^2} - \frac{\tau_1}{m+1} \,
\left( m \, \log \rho_1 + a^\downharpoonright_{{\tau_1 , \rho_1 , f^\upharpoonright_1,
f^\downharpoonright_1}} - \displaystyle \frac{3}{4} \, \log \tau \right) \in \mathbb R,
\]
and where $\bar V^\downharpoonright $ satisfies estimates of the form
(\ref{eq:6.321}) and (\ref{eq:6.32}). Again, to simplify the notations we have not
mentioned the parameters $\tau_1 , \rho_1 , f^\upharpoonright_1,
f^\downharpoonright_1 $ in the notation for $V^\upharpoonright, \bar
V^\upharpoonright, V^\downharpoonright$ and $ \bar V^\downharpoonright$.\\

\item[(iii)] Assume that we are given $\tau_2 >0$ satisfying
\[
\left|æ\tau_2 - \frac{\tau}{m+1} \right| \leq \kappa \, \tau^{3/2},
\]
and a function $f^\downharpoonright_2 \in \mathcal C^{2, \alpha} (S^1)$ which
satisfies both (H1), (H2) and
\[
\|æf_2^\downharpoonright \|_{\mathcal C^{2, \alpha} (S^1)} \leq \kappa \, \tau^{3/2} .
\]
The result of Proposition~\ref{pr:5.1} provides a constant mean
curvature (equal to $1$) surface $\mathfrak C_{\tau_2 ,
f^\downharpoonright_2}$ which is invariant under the action of $\mathcal S_3$,
the symmetry with respect to the horizontal plane $x_3=0$ and is also
invariant under the action of $\mathcal S_2$, the symmetry with respect to
the plane $x_2=0$. Moreover, close to its lower boundary, this surface
can be parameterized as the vertical graph of $x \longmapsto
U^\downharpoonright ( \tau^{-3/4} \, x )$ over $A^{ins}_\tau (0)$, where
\[
U^\downharpoonright (x) = d^\downharpoonright -
\tau_2 \, \log |x| + W^{\rm ins}_{f^\downharpoonright_2}
(x) + \bar U^\downharpoonright ( x ) ,
\label{eq:step1-1}
\]
where
\[
d^\downharpoonright : = - \tau_2 \, \log \left( \frac{2 \, \tau^{3/4} }{\tau_2} \right) \in
\mathbb R,
\]
and where $\bar U^\downharpoonright$ satisfies (\ref{eq:est-1bis})
and (\ref{eq:est-2bis}). To simplify the notations we have not
mentioned the data $\tau_2, f^\downharpoonright_2$ in the
notation for $\bar U^\downharpoonright$
and $\bar U^\downharpoonright$.
\end{itemize}

Let us emphasize that the functions $f^\downharpoonright_1,
f^\downharpoonright_2$ and $f^\upharpoonright , f^\upharpoonright_1$ are all
assumed to satisfy (H1). Hence they have no constant term in their Fourier series.
The function $f^\downharpoonright_2$ is also assumed to satisfy (H2). Now, the
functions $f^\upharpoonright$ and $f^\upharpoonright_1$ are assumed to be
invariant under the action of the dihedral group ${\rm Dih}^{(2)}_{m+1}$ and, as
already mentioned, this implies that both functions also satisfies (H2) since its
Fourier series not not contain any term of the form $z \cdot x$. Therefore,
$f^\downharpoonright_1$ is the only function which does not satisfy (H2). Since
$f^\downharpoonright_1$ is assumed to be invariant under the action of $\mathcal
S_2 $, we can decompose it as
\[
f^\downharpoonright_1 = \lambda_1 \, z_0 \cdot x +
f^{\downharpoonright , \perp}_1 ,
\]
where $\lambda_1 \in \mathbb R$ and where $f^{\downharpoonright , \perp}_1$
satisfies both (H1) and (H2).

We denote by $\mathfrak C_{\tau_2 , f^\downharpoonright_2, \rho_1}^{(0)}$
the surface $\mathfrak C_{\tau_2 , f^\downharpoonright_2}$ which has been
translated by $\rho_1 \, z_0$. For $j=1, \ldots, m$,
\[
\mathfrak C_{\tau_2 , f^\downharpoonright_2, \rho_1}^{(0)} : = \mathfrak C_{\tau_2 ,
f^\downharpoonright_2} + \rho_1 \, z_0 .
\]
Moreover, the image of $\mathfrak C_{\tau_2 , f^\downharpoonright_2, \rho_1
}^{(0)}$ under the rotation $(\mathcal R_{m+1})^j$ will be denoted by
$\mathfrak C_{\tau_2 , f^\downharpoonright_2, \rho_1}^{(j)}$
\[
\mathfrak C_{\tau_2 , f^\downharpoonright_2, \rho_1}^{(j)} : = (\mathcal R_{m+1})^j
\left( \mathfrak C_{\tau_2 , f^\downharpoonright_2, \rho_1}^{(0)}\right) .
\]
In particular, the collection of surfaces $\mathfrak C_{\tau_2 ,
f^\downharpoonright_2,
\rho_1}^{(0)}, \ldots, \mathfrak C_{\tau_2 , f^\downharpoonright_2, \rho_1}^{(m)}$
constitute $m+1$ constant mean curvature surfaces which are symmetric with
respect to the dihedral group ${\rm Dih}^{(3)}_{m+1}$.

Given $t_1 \in \mathbb R$ small enough, we denote by $\mathfrak S_{\tau_1 ,
\rho_1 , f^\upharpoonright_1, f^\downharpoonright_1, t_1}$ the surface
$\mathfrak S_{\tau_1 , \rho_1 , f^\upharpoonright_1,
f^\downharpoonright_1}$ which has been translated in the vertical direction by
$(1-c^\downharpoonright + d^\downharpoonright + t_1) \, e_3$
\[
\mathfrak S_{\tau_1 , \rho_1 , f^\upharpoonright_1, f^\downharpoonright_1, t_1}
: = \mathfrak S_{\tau_1 , \rho_1 , f^\upharpoonright_1, f^\downharpoonright_1}
+ (1-c^\downharpoonright + d^\downharpoonright + t_1) \, e_3 .
\]
This is a constant mean curvature surface which is symmetric with respect to
the dihedral group ${\rm Dih}^{(2)}_{m+1}$. Observe that the
lower boundaries of $\mathfrak C_{\tau_2 , f^\downharpoonright_2, \rho_1}^{(0)},
\ldots, \mathfrak C_{\tau_2 , f^\downharpoonright_2, \rho_1}^{(m)}$ are close to the
lower boundaries of $\mathfrak S_{\tau_1 , \rho_1 , f^\upharpoonright_1,
f^\downharpoonright_1, t_1}$.

Finally, given $t \in \mathbb R$ small enough, we denote by $\mathfrak
D_{\tau ,f, t}^+$ the surface $\mathfrak D_{\tau ,f, t}^+$ which has been
translated in the vertical direction by $(2 - c^\downharpoonright +
d^\downharpoonright - c^\upharpoonright + d^\upharpoonright + t_1 +t)
\, e_3$
\[
\mathfrak D_{\tau ,f, t}^+: = \mathfrak D_{\tau ,f}^++ (2 -
c^\downharpoonright + d^\downharpoonright - c^\upharpoonright +
d^\upharpoonright + t_1 +t) \, e_3 .
\]
This is a constant mean curvature surface which is symmetric with
respect to the dihedral group ${\rm Dih}^{(2)}_{m+1}$. Observe that the boundary
of $\mathfrak D_{\tau ,f, t}^+$ is close to the upper boundary of $\mathfrak
S_{\tau_1 , \rho_1 , f^\upharpoonright_1, f^\downharpoonright_1, t_1}$.

To complete the proof of the main theorem, it remains to adjust the free
parameters of our construction, namely $t, t_1 , \tau_1, \tau_2, \rho_1
\in \mathbb R$, and the functions $f^\downharpoonright_1,
f^\downharpoonright_2, f^\upharpoonright$ and $f^\upharpoonright_1$
defined on $S^1$, so that
\[
\mathfrak C_{\tau_2 , f^\downharpoonright_2, \rho_1}^{(0)} \sqcup
\ldots \sqcup \mathfrak C_{\tau_2 , f^\downharpoonright_2, \rho_1}^{(m)}
\sqcup \mathfrak S_{\tau_1 , \rho_1 , f^\upharpoonright_1,
f^\downharpoonright_1, t_1} \sqcup \mathfrak D_{\tau ,f, t}^+,
\]
constitute a $\mathcal C^1$ surface which can be extended by reflection
through the horizontal plane as a $\mathcal C^1$ surface which is complete,
non compact and has two ends of Delaunay type (asymptotic to a nodo\"{\i}d end).
Observe that the surface is invariant under the action of the dihedral group
${\rm Dih}^{(3})_{m+1}$ and that there is still one free parameter, namely $\tau$
which determines the Delaunay type end and hence the vertical flux of
the surface.

This surface is in fact piecewise smooth and has constant mean curvature equal
to $1$ away from the boundaries where the connected sum is performed.
Since all pieces have constant mean curvature identically equal to $1$, elliptic
regularity theory then implies that this surface is in fact a smooth surface. Indeed,
near one of the boundaries where the connected sum is performed, the
surface is a graph of a function, say $u^{ins}$ defined over $A^{ins}_\tau$
and another function, say $u^{out}$ defined over $A^{out}_\tau$. The functions
$u^{ins}$ and $u^{out}$ are $\mathcal C^{2, \alpha}$ and solve the mean
curvature equation
\begin{equation}
\frac{1}{2} \, {\rm div} \, \left( \frac{\nabla u}{\sqrt{1+|\nabla u|^2}} \right) = 1
\label{mce}
\end{equation}
on their respective domain of definition (for the sake of simplicity, we have
assumed that the mean curvature vector is upward pointing near the
boundary we are interested in). Moreover, $u^{ins} = u^{out}$ and
$\partial_r u^{ins} = \partial_r u^{out}$on $\partial A^{ins}_\tau \cap \partial
A^{out}_\tau$. This implies that the function $u$ defined on $A^{ins}_\tau \cup
A^{out}_\tau$ by $u : = u^{ins}$ on $A^{ins}_\tau$ and $u := u^{out}$ on
$A^{out}_\tau$ belongs to $\mathcal C^{1,1}$ and is a weak solution of (\ref{mce})
on $A^{ins}_\tau \cup A^{out}_\tau$. Elliptic regularity implies that $u$
is $\mathcal C^{2, \alpha}$ and hence the surface we have obtained is
a smooth constant mean curvature surface.

Therefore, to complete the proof, it remains to explain how to find $t, t_1 ,
\tau_1, \tau_2, \rho_1 \in \mathbb R$, and the functions $f^\downharpoonright_1,
f^\downharpoonright_2, f^\upharpoonright$ and $f^\upharpoonright_1$
defined on $S^1$, so that the following system of equations
is fulfilled
\begin{equation}
U^\upharpoonright - c^\upharpoonright + t = V^\upharpoonright
- 1 - d^\upharpoonright \qquad \mbox{and} \qquad
\partial_r \left( V^\upharpoonright - U^\upharpoonright \right) =0
\label{eq:7.1}
\end{equation}
on $S^1$ and
\begin{equation}
V^\downharpoonright + 1-c^\downharpoonright + t_1 = U^\downharpoonright
- d^\downharpoonright \qquad \mbox{and} \qquad
\partial_r \left( V^\downharpoonright - U^\downharpoonright \right) =0
\label{eq:7.2}
\end{equation}
on $S^1$. Recall that, even though this is not apparent in the notations, all functions
and constants depend on the parameters and boundary data. The rest of this
section is devoted to the proof that the above system has indeed a solution, provided
$\tau$ is small enough. We will prove the~:
\begin{proposition}
There exists $\kappa >0$ such that, for all $\tau >0$ small enough there exists
parameters $t, t_1, \tau_1, \tau_2, \rho_1$, and functions $f^\downharpoonright_1,
f^\downharpoonright_2, f^\upharpoonright , f^\upharpoonright_1$ defined on $S^1$
and satisfying the above symmetries and estimates, such that the system
(\ref{eq:7.1}) and (\ref{eq:7.2}) is satified.
\end{proposition}
\begin{proof} First we make use of the results of Propositions~\ref{pr:4.4},
Propositions~\ref{pr:5.1} and Propositions~\ref{pr:6.3} to get the expansion of the
functions $U^\upharpoonright , V^\upharpoonright , U^\downharpoonright $ and
$V^\downharpoonright $. Recalling that we have to restrict all those functions to
$S^1$, it is easy to check, using (\ref{eq:5.199}) and (\ref{eq:6.31}) that the first two
equations of the system we have to solve read
\begin{equation}
\left\{
\begin{array}{rllll}
t + f^\upharpoonright_1 - f^\upharpoonright & = & \bar
U^\upharpoonright - \bar V^\upharpoonright \\[3mm]
(\tau_1 - \tau) + \partial_r \left( W^{\rm out}_{f^\upharpoonright_1} - W^{\rm ins}_{
f^\upharpoonright} \right) & = & \partial_r \left( \bar V^\upharpoonright - \bar
U^\upharpoonright \right) ,
\end{array}
\right.
\label{eq:7.3}
\end{equation}
while, using (\ref{eq:5.200}) and (\ref{eq:6.33}), we see the next two equations are
given by
\begin{equation}
\left\{
\begin{array}{rllll}
t_1æ- \displaystyle \tau^{3/4}\, æ\left( \rho_1 - \frac{m}{m+1} \, \frac{\tau_1}{ 2
\rho_1}æ \right) \, z_0 \cdot x & +& f^\downharpoonright_1 - f^\downharpoonright_2
\\[3mm]
& = & \bar U^\downharpoonright - \bar V^\downharpoonright \\[3mm]
- \displaystyle \left( \frac{\tau_1}{m+1} - \tau_2 \right) - \displaystyle \tau^{3/4}\,
æ\left( \rho_1 - \frac{m}{m+1} \, \frac{\tau_1}{ 2 \rho_1}æ \right) \, z_0 \cdot x & + &
\partial_r \left(W^{\rm out}_{f^\downharpoonright_1} - W^{\rm ins}_{
f^\downharpoonright_2} \right) \\[3mm]
& = & \partial_r \left( \bar U^\downharpoonright - \bar V^\downharpoonright \right),
\end{array}
\right.
\label{eq:7.4}
\end{equation}
In writing this system one has to be a bit careful about the invariance of the functions
we are interested in. Indeed, in (\ref{eq:7.3}), all functions are invariant under the
action of ${\rm Dih}^{(2)}_{m+1}$., while in (\ref{eq:7.4}), all functions are invariant
under the action of the symmetry $\mathcal S_2$.

Let us denote by $\Pi^0$
the $L^2(S^1)$-orthogonal projection over the space of constant functions, $\Pi^1$
the $L^2(S^1)$-orthogonal projection over the space spanned by the function
$x \longmapsto z_0 \cdot x$ and let us denote by $\Pi^\perp$ denotes $L^2(S^1)
$-orthogonal projection over the orthogonal complement of the space spanned by
the constant function and the function $x \longmapsto z_0 \cdot x$.

We project this system over the $L^2(S^1)$-orthogonal complement of the constant
function and the function $x \longmapsto z_0 \cdot x$. We obtain the coupled system
\begin{equation}
\left\{
\begin{array}{rllll}
f^\upharpoonright_1 - f^\upharpoonright & = & \Pi^\perp \left( \bar
U^\upharpoonright - \bar V^\upharpoonright \right) \\[3mm]
\partial_r \left( W^{\rm out}_{f^\upharpoonright_1} - W^{\rm ins}_{f^\upharpoonright}
\right) & = & \Pi^\perp \partial_r \left( \bar V^\upharpoonright - \bar
U^\upharpoonright \right) \\[3mm]
f^{\downharpoonright, \perp}_1 - f^\downharpoonright_2 & = & \Pi^\perp \left( \bar
U^\downharpoonright - \bar V^\downharpoonright \right) \\[3mm]
\partial_r \left(W^{\rm out}_{f^{\downharpoonright, \perp}_1} - W^{\rm ins}_{
f^\downharpoonright_2} \right) & = & \Pi^\perp \partial_r \left( \bar
U^\downharpoonright - \bar V^\downharpoonright \right),
\end{array}
\right.
\label{eq:projorth}
\end{equation}
where we recall that we have decomposed $f^\downharpoonright_1
= \lambda_1 \, z_0 \cdot x + f^{\downharpoonright , \perp}_1$.

The projection of the system (\ref{eq:7.3})-(\ref{eq:7.4}) over the space of
constant functions leads to the coupled system
\begin{equation}
\left\{
\begin{array}{rllll}
t & = & \Pi^0 \left( \bar U^\upharpoonright - \bar V^\upharpoonright \right) \\[3mm]
\tau_1 - \tau & = & \Pi^0 \partial_r \left( \bar V^\upharpoonright - \bar
U^\upharpoonright \right) \\[3mm]
t_1 & = & \Pi^0 \left( \bar U^\downharpoonright - \bar V^\downharpoonright \right)
\\[3mm]
\displaystyle \tau_2 - \frac{\tau_1}{m+1} & = & \Pi^0 \partial_r \left( \bar
U^\downharpoonright - \bar V^\downharpoonright \right).
\end{array}
\right.
\label{eq:projorth0}
\end{equation}

Finally, the projection of the system (\ref{eq:7.3})-(\ref{eq:7.4}) over the space of
functions spanned by $x \longmapsto z_0\cdot x$ leads to the coupled system
\begin{equation}
\left\{
\begin{array}{rllll}
\left( \lambda_1 - \displaystyle \tau^{3/4}\,
æ\left( \rho_1 - \frac{m}{m+1} \, \frac{\tau_1}{ 2 \rho_1}æ \right) \right) \, z_0 \cdot x
& = & \Pi^1 \left( \bar U^\downharpoonright - \bar V\downharpoonright \right)
\\[3mm]
\left( -2 \, \lambda_1 - \displaystyle \tau^{3/4}\,
æ\left( \rho_1 - \frac{m}{m+1} \, \frac{\tau_1}{ 2 \rho_1}æ \right)\right) \, z_0 \cdot x & = &
\Pi^1 \partial_r \left( \bar U^\downharpoonright - \bar V^\downharpoonright \right).
\end{array}
\right.
\label{eq:projorth1}
\end{equation}
To obtain the second equation, we have used the fact that
\[
W^{\rm out}_{f^{\downharpoonright}_1} = \lambda_1 \, \frac{z_0 \cdot x}{
|x|^2} + W^{\rm out}_{f^{\downharpoonright, \perp}_1} .
\]

Observe that the right hand sides of (\ref{eq:projorth}), (\ref{eq:projorth0}) and
(\ref{eq:projorth1}) does not depend on $t$ and $t_1$. Hence, the first and third
equations in (\ref{eq:projorth0}) will give us the values of $t$ and $t_1$, once the rest
of the equations are solved.

For all $\tau$ small enough, we will solve (\ref{eq:projorth}) using some fixed point
theorem for contraction mappings to obtain a solution $(f^\upharpoonright,
f^\upharpoonright_1 , f^{\downharpoonright, \perp}_1, f^{\downharpoonright}_2)$
continuously depending on the parameters $\tau_1, \tau_2, \rho_1, \lambda_1$ (and
$\tau$). Then, we introduce the corresponding solution in (\ref{eq:projorth0}) and
(\ref{eq:projorth1}) to get a nonlinear system in $\tau_1, \tau_2$ and $\rho_1$, which
we will solve using Browder's fixed point theorem.

To begin with, we explain how (\ref{eq:projorth}) can be rewritten in diagonal form.
This makes use of the following result whose proof can be found, for example, in
\cite{Maz-Pac-1}~:
\begin{proposition}
The operator
\[
\mathcal C^{2, \alpha} (S^1)^\perp \ni f \longmapsto \partial_r \left(W^{\rm ins}_{f} -
W^{\rm out}_{f} \right)_{| r=1} \in \mathcal C^{1, \alpha} (S^1)^\perp
\]
is an isomorphism. Here $\mathcal C^{k, \alpha} (S^1)^\perp$ denote the image of
$\mathcal C^{k, \alpha} (S^1)$ under $\Pi^\perp$.
\end{proposition}
\begin{proof}
The Fourier decomposition of a function $f \in \mathcal C^{k, \alpha} (S^1)^\perp$ is
given by
\[
f (\theta) = \sum_{n\neq 0, \pm 1} f_n \, e^{i n \theta}
\]
in which case
\[
W^{\rm out}_{f} = \sum_{n\neq 0, \pm 1} f_n \, r^{-|n|} \, e^{i n\theta}, \quad \mbox{and}
\quad W^{\rm ins}_{f} = \sum_{n \neq 0, \pm 1} f_n \, r^{|n|} \, e^{i n\theta},
\]
Therefore,
\[
\partial_r \left(W^{\rm ins}_{f} - W^{\rm out}_{f} \right)_{| r=1} = 2 \, \sum_{n \neq 0,
\pm 1} f_n \, |n| \, e^{i n\theta},
\]
is equal to twice the Dirichlet to Neumann map for the Laplace operator in the unit
disc. This is a well defined, self-adjoint, first order elliptic operator which is
injective and elliptic regularity theory implies that it is an isomorphism.
\end{proof}

Using this result, the system (\ref{eq:projorth}) can be rewritten as
\[
\left (f^\upharpoonright, f^\upharpoonright_1 , f^{\downharpoonright, \perp}_1,
f^{\downharpoonright}_2 \right) =
\mathbb N^\perp_{\tau_1, \tau_2, \rho_1, \lambda_1 } (f^\upharpoonright,
f^\upharpoonright_1 f^{\downharpoonright, \perp}_1,
f^{\downharpoonright}_2) ,
\]
where the nonlinear operator $\mathbb N^\perp_{\tau_1, \tau_2, \rho_1, \lambda_1 }
$ satisfies
\begin{equation}
\| \mathbb N^\perp_{\tau_1, \tau_2, \rho_1, \lambda_1 } (f^\upharpoonright,
f^\upharpoonright_1 , f^{\downharpoonright, \perp}_1,
f^{\downharpoonright}_2)\|_{(\mathcal C^{2,
\alpha} (S^1))^{4}} \leq \, C \, \tau^{3/2}
\label{eq:sa}
\end{equation}
for some constant $C >0$ independent of $\kappa >0$, provided $\tau$ is chosen
small enough. This last estimate follows directly from (\ref{eq:est-1})
in Proposition~\ref{pr:4.4}, (\ref{eq:est-1bis}) in Proposition~\ref{pr:5.1} and
(\ref{eq:6.321}) in Proposition~\ref{pr:6.3}. Moreover, thanks to (\ref{eq:est-2}),
(\ref{eq:est-2bis}) and (\ref{eq:6.32}), provided $\kappa >0$ is fixed larger than the
constant $C$ which appears in (\ref{eq:sa}), we can use a fixed point theorem for
contraction mapping in the ball of radius $\kappa \, \tau^{3/2}$ in $\left( \Pi^\perp
\mathcal C^{2, \alpha} (S^1)\right)^4$ to get the existence of a solution of
(\ref{eq:sa}), for all $\tau >0$ small enough. This solution depends
continuously on $\tau_1, \tau_2, \rho_1$ and $\lambda_1$, since $\mathbb N^
\perp_{\tau_1, \tau_2, \rho_1, \lambda_1 }$ does (observe that $\mathbb N^
\perp_{\tau_1, \tau_2, \rho_1, \lambda_1 } $
depends implicitly on $\tau$). We now insert this solution in (\ref{eq:projorth0})
and (\ref{eq:projorth1}). With simple manipulations, we conclude that it remains to
solve the nonlinear system
\begin{equation}
\left( \tau_1 - \tau , \tau_2 - \frac{\tau}{m+1} , \tau^{3/4} \,æ\left( \rho_1 - \frac{m}{m+1}
\, \frac{\tau}{ 2 \rho_1}æ \right), æ \lambda_1 \right) = \mathbb N^0 ( \tau_1, \tau_2,
\rho_1 , \lambda_1) ,
\label{eq:la}
\end{equation}
where $\mathbb N^0$ satisfies
\[
\| \mathbb N^0 ( \tau_1, \tau_2, \rho_1 , \lambda_1)\|_{\mathbb R^4} \leq \, C \,
\tau^{3/2}
\]
for some constant $C >0$ independent of $\kappa >0$, provided $\tau$ is chosen
small enough. Moreover, $\mathbb N^0$ depends continuously on the parameters
$\tau_1, \tau_2, \rho_1$ and $ \lambda_1$ (observe that $\mathbb N^0$ depends
implicitly on $\tau$). The equation (\ref{eq:la}) can then be solved using a simple
degree argument (Browder's fixed point theorem). This completes the proof of the
result.
\end{proof}

\section{Appendix 1}

We discuss the elementary result in the theory of second order ordinary differential
equations which is used at the end of the proof of Proposition~\ref{pr:linres100}.
Assume that we are given a function $s \longmapsto p(s)$ which is periodic (say of
period $S >0$). Further assume that the homogeneous problem $ (\partial_s^2 +
p )\, w^+ =0$ has a nontrivial periodic solution of period $S$. Without loss of
generality, we can assume that $ w^+(0)=1$ and $\partial_s w^+(0) =0$ (just
choose the origin so that $0$ coincides with a point where $w^+$ achieves its
maximum). Let $w^-$ be the unique solution of $(\partial_s^2 + p)\, w^- =0$ such
that $w^-(0) = 0$ and $\partial_s w^-(0) = 1$. The Wronskian of $w^+$ and $w^-$
being constant, we conclude that
\[
\partial_s w^-(S) = \partial_s w^-(S) w^+(S) - \partial_s w^+ (S) w^-(S) =
\partial_s w^-(0) w^+(0) - \partial_s w^+ (0)w^-(0) = 1.
\]

We define
\[
v(s) : = w^- (S+s) - w^-(S) \, w^+(s) .
\]
It is clear that $v$ is a solution $(\partial_s^2 + p) \, v = 0$ and further
observe that $\partial_s v(0) =1$ and $v(0) = 0$. Therefore, $v=w^-$. This proves
that
\[
w^- (S+s) = w^-(s) + w^-(S) \, w^+(s),
\]
and hence $w^-$ is at most linearly growing in the sense that $
|w^-(s)| \leq C \, (1+ |s|)$ for some constant $ C>0$.

\end{document}